\documentclass{article}
\PassOptionsToPackage{numbers, compress}{natbib}

\usepackage{subfloat}
\usepackage[preprint]{neurips_2022}

\usepackage{selectp}

\usepackage{algpseudocode}
\usepackage{algorithm}
\usepackage[utf8]{inputenc} 
\usepackage[T1]{fontenc}    
\usepackage{hyperref}       
\usepackage{url}            
\usepackage{booktabs}       
\usepackage{amsfonts}       
\usepackage{nicefrac}       
\usepackage{microtype}      
\usepackage{xcolor}         

\usepackage{graphicx}
\usepackage{amsmath, amsthm, amssymb}
\usepackage{bm}

\newtheorem{defi}{Definition}
\newtheorem{thm}{Theorem}

\newtheorem{lem}{Lemma}

\newtheorem{rem}{Remark}

\newcommand{\blue}{\textcolor{black}} 

\newcommand{\bbeta}{\bm{\beta}}
\newcommand{\bgamma}{\bm{\gamma}}
\newcommand{\bxi}{\bm{\xi}}
\newcommand{\bx}{\bm{x}}

\newcommand{\bz}{\bm{z}}
\newcommand{\bw}{\bm{w}}

\newcommand{\mR}{\mathbb{R}}
\newcommand{\mB}{\mathbb{B}}
\newcommand{\mP}{\mathbb{P}}
\newcommand{\mQ}{\mathbb{Q}}
\newcommand{\mE}{\mathbb{E}}

\usepackage{dsfont}
\newcommand*\diff{\mathop{}\!\mathrm{d}}

\newcommand{\indicate}[1]{\mathds{1}\small{[#1]}}

\usepackage{caption}

\title{Wasserstein Logistic Regression with Mixed Features}

\author{%
  Aras Selvi \quad Mohammad Reza Belbasi \quad Martin B. Haugh \quad  Wolfram Wiesemann\\
  Imperial College Business School, Imperial College London, United Kingdom\\
  \texttt{\{a.selvi19, r.belbasi21, m.haugh, ww\}@imperial.ac.uk} \\
}

\begin{document}

\maketitle

\begin{abstract}
    Recent work has leveraged the popular distributionally robust optimization paradigm to combat overfitting in classical logistic regression. While the resulting classification scheme displays a promising performance in numerical experiments, it is inherently limited to numerical features. In this paper, we show that distributionally robust logistic regression with mixed (\emph{i.e.}, numerical and categorical) features, despite amounting to an optimization problem of exponential size, admits a polynomial-time solution scheme. We subsequently develop a practically efficient column-and-constraint approach that solves the problem as a sequence of polynomial-time solvable exponential conic programs. Our model retains many of the desirable theoretical features of previous works, but---in contrast to the literature---it does not admit an equivalent representation as a regularized logistic regression, that is, it represents a genuinely novel variant of logistic regression. We show that our method outperforms both the unregularized and the regularized logistic regression on categorical as well as mixed-feature benchmark instances.
\end{abstract}

\section{Introduction}\label{sec:introduction}

Consider a data set $(\bm{x}^i, y^i)_{i = 1}^N$ with feature vectors $\bm{x}^i$ and associated binary labels $y^i \in \{ -1, 1 \}$. Classical logistic regression assumes that the labels depend probabilistically on the features via
\begin{equation*}
    \mbox{Prob} (y \mid \bx) = \left[ 1 + \exp (-y \cdot [\beta_0 + \bbeta^\top \bx]) \right]^{-1},
\end{equation*}
where the parameters $(\beta_0, \bbeta) \in \mathbb{R}^{1+n}$ are estimated from the empirical risk minimization problem
\begin{equation*}
    \begin{array}{l@{\quad}l}
        \displaystyle \mathop{\text{minimize}}_{(\beta_0, \bbeta)} & \displaystyle \frac{1}{N} \sum_{i=1}^N l_{\bbeta}(\bx^i, y^i) \\[6mm]
        \displaystyle \text{subject to} & \displaystyle (\beta_0, \bbeta) \in \mathbb{R}^{n+1}
    \end{array}
\end{equation*}
with the \emph{log-loss} function $l_{\bbeta} (\bx, y) := \log \left( 1 + \exp \left( -y \cdot [\beta_0 + \bm{\bbeta}^\top \bx] \right) \right)$. 
Its compelling performance across many domains, the availability of mature and computationally efficient algorithms as well as its interpretability have all contributed to the widespread adoption of logistic regression \cite{bishop2006pattern, hastie2009elements, murphy2022}.

Similar to other machine learning models, logistic regression is prone to overfitting, especially when the number of training samples is small relative to the number of considered features. Moreover, logistic regression can be sensitive to erroneous feature and label values as well as distribution shifts under which the training and test sets stem from different distributions. In recent years, distributionally robust (DR) optimization \cite{BTEGN09:rob_opt, BdH22:rob_opt} has been proposed to simultaneously address these challenges. To this end, the DR optimization paradigm models a machine learning task as a zero-sum game between the decision maker, who seeks to obtain the most accurate model (\emph{e.g.}, the parameter vector $\bm{\beta}$ in a logistic regression), and a fictitious adversary who observes the decision maker's model and subsequently selects the worst data-generating distribution in the vicinity of the empirical distribution formed from the available training data. In this context, the similarity of distributions is commonly measured in terms of moment bounds \cite{DY10:distr_rob_opt, WKS14:DRCO}, $\phi$-divergences such as the Kullback-Leibler divergence \cite{HH12:kl_divergence, L19:kl_divergence} or the Wasserstein distance \cite{Gao_Kleywegt_2016, med17}. It has been observed that in many cases, the resulting DR machine learning models are equivalent to regularized versions of the same models, in which case distributionally robustness provides a new perspective on regularizers that are often selected \emph{ad hoc} \cite{kmns19, XCM09:regularization}. Most importantly, DR machine learning models often admit computationally efficient reformulations as finite-dimensional convex optimization problems that can be solved in polynomial time with off-the-shelf solvers.

In this work we study a DR variant of the logistic regression where the distance between the empirical distribution and the unknown true data-generating distribution is measured by the popular Wasserstein (also known as Kantorovich-Rubinstein or earth mover's) distance \cite{Gao_Kleywegt_2016, med17}. This problem has been first studied by \cite{NIPS2015}, who show that the resulting DR logistic regression problem admits an equivalent reformulation as a polynomial-size convex optimization problem \emph{if all features are numerical}. Applying similar techniques to the mixed-feature logistic regression problem, which appears to be most prevalent in practice, would result in an exponential-size convex optimization problem whose na\"ive solution with the methods of \cite{NIPS2015} does not scale to interesting problem sizes (\emph{cf.}~Section~\ref{sec:numericals}). Our contributions may be summarized as follows. \vspace{-.15cm}
\begin{enumerate}
    \item[\emph{(i)}] On the \emph{theoretical side}, we show that the complexity of DR mixed-feature regression crucially relies on the selected loss function. In particular, for the log-loss function employed in logistic regression, the problem---despite its natural representation as an exponential-size convex optimization problem---admits a polynomial-time solution scheme. We also show that in stark contrast to earlier variants of the problem, our mixed-feature regression does \emph{not} admit an equivalent representation as a regularized problem. \blue{This provides compelling evidence that the DR logistic regression problem with mixed features is fundamentally different to the DR problem with only numerical features.}
    \item[\emph{(ii)}] On the \emph{computational side}, we propose a column-and-constraint scheme that solves the DR mixed-feature logistic regression as a sequence of polynomial-time solvable exponential conic programs. We show that the key step of our procedure, the identification of the most violated constraint, can be implemented efficiently for a broad range of metrics, despite its natural representation as a combinatorial optimization problem. \blue{Indeed, identifying the most violated constraint by brute force is out of the question, while the standard approach \cite{ZENG2013457} of solving the most violated constraint problem would have an unacceptably high runtime.} 
    \item[\emph{(iii)}] On the \emph{numerical side}, we show that our column-and-constraint scheme drastically reduces computation times over a na\"ive monolithic implementation of the regression problem. We also show that our model performs favorably on standard categorical and mixed-feature benchmark instances when compared against classical and regularized logistic regression.
\end{enumerate}

The literature on DR machine learning under the Wasserstein distance is vast and rapidly growing. Recent works have considered, among others, the use of Wasserstein DR models in multi-label learning \cite{NIPS2015_a9eb8122}, generative adversarial networks \cite{pmlr-v70-arjovsky17a, NEURIPS2019_f35fd567} and the generation of adversarial examples \cite{pmlr-v97-wong19a}, density estimation \cite{pmlr-v99-weed19a} and learning Gaussian mixture models \cite{Kolouri_2018_CVPR}, graph-based semi-supervised learning \cite{pmlr-v32-solomon14}, supervised dimensionality reduction \cite{FCCR18:wasserstein} and reinforcement learning \cite{DBLP:journals/corr/abs-1907-13196}. \blue{DR has also been found to help alleviate problems with overfitting \cite{Esfahani-Kuhn-2018,kmns19}, label uncertainty \cite{kmns19} and distribution shifts \cite{pmlr-v139-taskesen21a,NEURIPS2021_f8689009}.} We refer to \cite{kmns19} for a recent review of the literature. Our work is most closely related to \cite{NIPS2015}, and we compare our findings with the results of that work in Section~\ref{sec:mixed_feature_lr}.

We proceed as follows. Section~\ref{sec:mixed_feature_lr} defines and analyzes the mixed-feature DR logistic regression problem, which is solved in Section~\ref{sec:cac_generation} via column-and-constraint generation. We report numerical results in Section~\ref{sec:numericals}. Auxiliary material and all proofs are relegated to the appendix.

\textbf{Notation.} We define $\mathbb{B} = \{ 0, 1 \}$ and $[N] = \{ 1, \ldots, N \}$ for $N \in \mathbb{N}$. The set of all probability distributions supported on a set $\Xi$ is denoted by $\mathcal{P}_0 (\Xi)$, while the Dirac distribution placing unit probability mass on $\bm{x} \in \mathbb{R}^n$ is denoted by $\delta_{\bm{x}} \in \mathcal{P}_0 (\mathbb{R}^n)$. The indicator function $\indicate{\mathcal{E}}$ attains the value $1$ ($0$) whenever the expression $\mathcal{E}$ is (not) satisfied. \blue{Finally, we use Roman $\mathrm{d}$ to denote differentials and to distinguish them from the $d$ used to denote distances. }

\section{Mixed-Feature DR Logistic Regression}\label{sec:mixed_feature_lr}

Section~\ref{sec:mixed_feature_lr:exponential_reform} derives an exponential-size convex programming formulation of the DR logistic regression problem that serves as the basis of our analysis. Section~\ref{sec:comparison_to_continuous} compares our formulation with a na\"ive model that treats categorical features as continuous ones. Section~\ref{sec:mixed_feature_lr:complexity}, finally, shows that our formulation, despite its exponential size, can be solved in polynomial time due to the benign structure of the log-loss function. It also establishes that in stark contrast to the literature, our formulation does not reduce to a regularized non-robust logistic regression, that is, it constitutes a genuinely novel variant of the logistic regression problem.

Our formulation enjoys strong finite-sample and asymptotic performance guarantees from the literature. Moreover, despite its exponential size, our model always accommodates worst-case distributions that exhibit a desirable sparsity pattern. We relegate these results to Appendix~A.

\subsection{Formulation as an Exponential-Size Convex Program}\label{sec:mixed_feature_lr:exponential_reform}

From now on, we consider a mixed-feature data set $\{ \bm{\xi}^i := (\bm{x}^i, \bm{z}^i, y^i) \}_{i \in [N]}$ with $n$ numerical features $\bm{x}^i = (x^i_1, \ldots, x^i_n) \in \mathbb{R}^n$, $m$ categorical features $\bm{z}^i = (\bm{z}^i_1, \ldots, \bm{z}^i_m) \in \mathbb{C} (k_1) \times \ldots \times \mathbb{C} (k_m)$ and a binary label $y^i \in \{ -1, + 1 \}$. Here, $\mathbb{C} (s) = \{ \bm{z} \in \mathbb{B}^{s - 1} \, : \, \sum_{j \in [s-1]} z_j \leq 1 \}$ for $s \in \mathbb{N} \setminus \{ 1 \}$ represents the one-hot encoding of a categorical feature with $s$ possible values; in particular, $\mathbb{C} (2) = \mathbb{B}$ encodes a binary feature. We denote by $\mathbb{C} = \mathbb{C} (k_1) \times \ldots \times \mathbb{C} (k_m)$ and $\Xi = \mathbb{R}^n \times \mathbb{C} \times \{ -1, +1 \}$ the support of the categorical features as well as the data set, respectively, and we let $k = k_1 + \ldots + k_m - m$ be the number of slopes used for the categorical features.

If we had access to the true data-generating distribution $\mathbb{P}^0 \in \mathcal{P}_0 (\Xi)$, we would solve the (non-robust) logistic regression problem
\begin{equation*}
    \begin{array}{l@{\quad}l}
        \displaystyle \mathop{\text{minimize}}_{\bm{\beta}} & \displaystyle \mathbb{E}_{\mathbb{P}^0} \left[ l_{\bm{\beta}} (\bm{x}, \bm{z}, y) \right] \\
        \displaystyle \text{subject to} & \displaystyle \bm{\beta} = (\beta_0, \bm{\beta}_{\text{N}}, \bm{\beta}_{\text{C}}) \in \mathbb{R}^{1+n + k },
    \end{array}
\end{equation*}
where $\mathbb{E}_{\mathbb{P}^0}$ denotes the expectation under $\mathbb{P}^0$, and where the log-loss function $l_{\bm{\beta}}$ now takes the form
\begin{equation*}
    l_{\bm{\beta}} (\bm{x}, \bm{z}, y) := \log \left( 1 + \exp \left[ -y \cdot \left( \beta_0 + \bm{\beta}_{\text{N}}{}^\top \bm{x} + \bm{\beta}_{\text{C}}{}^\top \bm{z} \right) \right] \right)
\end{equation*}
to account for the presence of categorical features. Since the distribution $\mathbb{P}^0$ is unknown in practice, the empirical risk minimization problem replaces $\mathbb{P}^0$ with the empirical distribution $\widehat{\mathbb{P}}_N := \frac{1}{N} \sum_{i = 1}^N \delta_{\bm{\xi}^i}$ that places equal probability mass on all observations $\{ \bm{\xi}^i \}_{i \in [N]}$. Standard arguments show that when these observations are i.i.d., the empirical risk minimization problem recovers the logistic regression under $\mathbb{P}^0$ as $N \longrightarrow \infty$. In practice, however, data tends to be scarce, and the empirical risk minimization problem exhibits an `optimism bias' that is also known as overfitting \cite{bishop2006pattern, hastie2009elements, murphy2022}, the error maximization effect of optimization \cite{demiguel2009portfolio, michaud1989markowitz} or the optimizer's curse \cite{smith2006optimizer}.

DR logistic regression combats the aforementioned overfitting phenomenon by solving the semi-infinite optimization problem
\begin{equation}\label{eq:the_mother_of_all_problems}
    \begin{array}{l@{\quad}l}
        \displaystyle \mathop{\text{minimize}}_{\bm{\beta}} & \displaystyle \sup_{\mathbb{Q} \in \mathfrak{B}_\epsilon (\widehat{\mathbb{P}}_N)} \; \mathbb{E}_\mathbb{Q} \left[ l_{\bm{\beta}} (\bm{x}, \bm{z}, y) \right] \\[5mm]
        \displaystyle \text{subject to} & \displaystyle \bm{\beta} = (\beta_0, \bm{\beta}_{\text{N}}, \bm{\beta}_{\text{C}}) \in \mathbb{R}^{1 + n + k},
    \end{array}
\end{equation}
where the \emph{ambiguity set} $\mathfrak{B}_\epsilon (\widehat{\mathbb{P}}_N)$ contains all distributions $\mathbb{Q}$ in a (soon to be defined) vicinity of the empirical distribution $\widehat{\mathbb{P}}_N$, and where the expectation is taken with respect to $\mathbb{Q}$. Problem~\eqref{eq:the_mother_of_all_problems} can be interpreted as a zero-sum game between the decision maker, who chooses a logistic regression model parameterized by $\bm{\beta}$, and a fictitious adversary that observes $\bm{\beta}$ and subsequently chooses the `worst' distribution (in terms of the incurred log-loss) from $\mathfrak{B}_\epsilon (\widehat{\mathbb{P}}_N)$. Contrary to the classical non-robust logistic regression, problem~\eqref{eq:the_mother_of_all_problems} guarantees to \emph{over}estimate the log-loss incurred by $\bm{\beta}$ under the unknown true distribution $\mathbb{P}^0$ as long as $\mathbb{P}^0$ is contained in $\mathfrak{B}_\epsilon (\widehat{\mathbb{P}}_N)$. On the other hand, problem~\eqref{eq:the_mother_of_all_problems} recovers the empirical risk minimization problem when the ambiguity set $\mathfrak{B}_\epsilon (\widehat{\mathbb{P}}_N)$ approaches a singleton set that only contains the empirical distribution $\widehat{\mathbb{P}}_N$. Note that problem~\eqref{eq:the_mother_of_all_problems} is convex as the convexity of the logistic regression objective $\mathbb{E}_\mathbb{Q} \left[ l_{\bm{\beta}} (\bm{x}, \bm{z}, y) \right]$ is preserved under the supremum operator. That said, the problem typically constitutes a semi-infinite program as {\color{black}it} comprises finitely many decision variables but---if the embedded supremum in the objective is brought to the constraints via an epigraph reformulation---infinitely many constraints whenever the ambiguity set harbors infinitely many distributions. As such, it is not obvious how the problem can be solved efficiently.

In this paper, we choose as ambiguity set the Wasserstein ball $\mathfrak{B}_\epsilon (\widehat{\mathbb{P}}_N) := \{ \mathbb{Q} \in \mathcal{P}_0 (\Xi) \, : \, \mathrm{W} (\mathbb{Q}, \widehat{\mathbb{P}}_N) \leq \epsilon \}$ of radius $\epsilon > 0$ that is centered at the empirical distribution $\widehat{\mathbb{P}}_N$.

\begin{defi}[Wasserstein Distance]\label{def:wasserstein}
    The type-1 \emph{Wasserstein} (\emph{Kantorovich-Rubinstein}, or \emph{earth mover's}) \emph{distance} between two distributions $\mP \in \mathcal{P}_0 (\Xi)$ and $\mQ \in \mathcal{P}_0 (\Xi)$ is defined as
    \begin{equation}\label{eq:Wass}
        \mspace{-11mu}
        \mathrm{W} (\mQ, \mP) := \inf_{\Pi \in \mathcal{P}_0 (\Xi^2)} \left\{ \int_{\Xi^2} d(\bxi,\bxi') \, \Pi(\mathrm{d} \bxi, \mathrm{d} \bxi')  \ : \  \Pi(\mathrm{d} \bxi, \Xi) = \mQ(\mathrm{d} \bxi), \, \Pi(\Xi, \mathrm{d} \bxi') = \mP(\mathrm{d} \bxi') \right\},
    \end{equation}
    where $\bxi = (\bx,\bz,y) \in \Xi$ and $\bxi' = (\bx',\bz',y') \in \Xi$, while $d(\bxi,\bxi')$ is the ground metric on $\Xi$.
\end{defi}

The Wasserstein distance can be interpreted as the minimum cost of moving $\mQ$ to $\mP$ when $d(\bxi,\bxi')$ is the cost of moving a unit mass from $\bxi$ to $\bxi'$. The Wasserstein radius $\epsilon$ thus imposes a budget on the transportation cost that the adversary can spend on perturbing the empirical distribution $\widehat{\mathbb{P}}_N$. 

We next define the ground metric $d$ that we use throughout the paper.

\begin{defi}[Ground Metric] \label{def:metric}
    We measure the distance between two data-points $\bxi = (\bx,\bz,y) \in \Xi$ and $\bxi'=(\bx',\bz',y') \in \Xi$ with $\bm{z} = (\bm{z}_1, \ldots, \bm{z}_m)$ and $\bm{z}' = (\bm{z}'_1, \ldots, \bm{z}'_m)$ as
    \begin{subequations}\label{eq:ground_metric}
        \begin{equation}\label{eq:ground_metric:overall}
            d (\bxi,\bxi') := \lVert \bx-\bx' \rVert + d_{\text{\emph{C}}} (\bz,\bz') + \kappa \cdot \indicate{y \neq y'},
        \end{equation}
        where $\lVert \cdot \rVert$ is any rational norm on $\mR^n$, $\kappa > 0$ and the metric $d_{\text{\emph{C}}}$ on $\mathbb{C}$ satisfies
        \begin{equation}\label{eq:ground_metric:categorical}
            d_{\text{\emph{C}}} (\bz,\bz') := \Big( \sum_{i \in [m]} \indicate{\bm{z}_i \neq \bm{z}'_i} \Big) ^{1 / p}
            \quad \text{for some } p > 0.
        \end{equation}
    \end{subequations}
\end{defi}

Intuitively, the ground metric $d$ measures the distance of the numerical features $\bm{x}$ and $\bm{x}'$ by the norm distance $\lVert \bx-\bx' \rVert$, whereas the distance of the categorical features $\bm{z}$ and $\bm{z}'$ is measured by (a power of) the number of discrepancies between $\bm{z}$ and $\bm{z}'$. Likewise, any discrepancy between the labels $y$ and $y'$ is accounted for by a constant $\kappa$, which allows for scaling between the features and the labels.

We first extend the convex optimization model of \cite{NIPS2015} for the DR continuous-feature logistic regression to an exponential conic model that accommodates for categorical features.

\begin{thm}[Exponential Conic Representation]\label{thm:coneheads}
    The DR logistic regression problem~\eqref{eq:the_mother_of_all_problems} admits the equivalent reformulation
    \begin{equation}\label{eq:cone_head_reformulation}
        \mspace{-40mu}
        \begin{array}{l@{\quad}l}
            \displaystyle \mathop{\text{\emph{minimize}}}_{\bbeta, \lambda, \bm{s}, \bm{u}^\pm, \bm{v}^\pm} & \displaystyle \lambda \epsilon + \dfrac{1}{N}\sum_{i \in [N]} s_i \\[6mm]
            \displaystyle \raisebox{4.5mm}{$\mspace{6mu}$ \text{\emph{subject to}}} &
            \mspace{-10mu} \left. \begin{array}{l}
                \displaystyle u_{i,\bz}^{+} + v_{i,\bz}^{+} \leq 1 \\
                \displaystyle (u_{i,\bz}^{+}, 1, -s_i - \lambda d_\text{\emph{C}}(\bz,\bz^i)) \in \mathcal{K}_{\exp} \\
                \displaystyle (v_{i,\bz}^{+}, 1, -y^i\bbeta_{\text{\emph{N}}}{}^\top \bx^i -y^i\bbeta_{\text{\emph{C}}}{}^\top \bz - y^i \beta_0 -s_i - \lambda d_\text{\emph{C}}(\bz,\bz^i)) \in \mathcal{K}_{\exp}
            \end{array}
            \mspace{29mu} \right\} \displaystyle \forall i \in [N], \; \forall \bz \in \mathbb{C} \\[6mm]
            & \mspace{-10mu} \left. \begin{array}{l}
                \displaystyle u_{i,\bz}^{-} + v_{i,\bz}^{-} \leq 1 \\
                \displaystyle (u_{i,\bz}^{-}, 1, -s_i - \lambda d_\text{\emph{C}}(\bz,\bz^i) - \lambda \kappa) \in \mathcal{K}_{\exp} \\
                \displaystyle (v_{i,\bz}^{-}, 1, y^i\bbeta_{\text{\emph{N}}}{}^\top \bx^i + y^i\bbeta_{\text{\emph{C}}}{}^\top \bz + y^i \beta_0 -s_i - \lambda d_\text{\emph{C}}(\bz,\bz^i) - \lambda \kappa) \in \mathcal{K}_{\exp}
            \end{array}
            \right\} \displaystyle \forall i \in [N], \; \forall \bz \in \mathbb{C} \\[6mm]
            & \displaystyle ||\bbeta_{\text{\emph{N}}}||_* \leq \lambda \\[1.5mm]
            & \displaystyle \bm{\beta} = (\beta_0, \bm{\beta}_{\text{\emph{N}}}, \bm{\beta}_{\text{\emph{C}}}) \in \mathbb{R}^{1+n + k }, \;\; \lambda \geq 0, \;\; \bm{s} \in \mathbb{R}^N \\
            & \displaystyle (u_{i,\bz}^{+}, v_{i,\bz}^{+}, u_{i,\bz}^{-}, v_{i,\bz}^{-}) \in \mathbb{R}^4, \, i \in [N] \text{\emph{ and }} \bz \in \mathbb{C}
        \end{array}
    \end{equation}
    as a finite-dimensional exponential conic program, where $\mathcal{K}_{\exp}$ denotes the \emph{exponential cone}
\begin{align*}
    \mathcal{K}_{\exp} := \mathrm{cl} \left(\left\{(a,b,c) \ : \ a \geq b \cdot \exp(c / b), \ a > 0, \ b > 0 \right\}\right) \subset \mathbb{R}^3.
\end{align*}
\end{thm}

\blue{The $s_i$'s and $\lambda$ that appear in (\ref{eq:cone_head_reformulation}) are dual variables that arise from dual of the inner (\textit{i.e.}, $\sup$) problem in (\ref{eq:the_mother_of_all_problems}). For each $\bz \in \mathbb{C}$ and $i \in [N]$, $(u_{i,\bz}^{+}, v_{i,\bz}^{+}, u_{i,\bz}^{-}, v_{i,\bz}^{-})$ is a vector of auxiliary variables that we use to model softplus constraints that arise from our intermediate analysis of problem (\ref{eq:the_mother_of_all_problems}) (Appendix~E).}

Problem~\eqref{eq:cone_head_reformulation} is finite-dimensional and convex. Moreover, since exponential conic programs admit a self-concordant barrier function, they can be solved in polynomial time relative to their input size \cite{N18:convex}. Due to the presence of categorical features in our regression problem, however, problem~\eqref{eq:cone_head_reformulation} comprises exponentially many variables and constraints, and a na\"ive solution of~\eqref{eq:cone_head_reformulation} would therefore require an \emph{exponential} amount of time. The exponential size of problem~\eqref{eq:cone_head_reformulation} renders our formulation fundamentally different from the model of \cite{NIPS2015}, and it significantly complicates the solution.

\subsection{Comparison with Continuous-Feature DR Logistic Regression}\label{sec:comparison_to_continuous}

Our formulation~\eqref{eq:cone_head_reformulation} of the mixed-feature DR logistic regression~\eqref{eq:the_mother_of_all_problems} accounts for categorical features $\bm{z} \in \mathbb{C}$ at the expense of an exponential number of variables and constraints. It is therefore tempting to treat the categorical features as continuous ones and directly apply the continuous-feature-only reformulation of the DR logistic regression~\eqref{eq:the_mother_of_all_problems} proposed by \cite{NIPS2015}. In the following, we argue that such a reformulation would hedge against nonsensical worst-case distributions that in turn lead to overly conservative regression models. To see this, consider a stylized setting where the data set $(z^i, y^i)_{i \in [N]}$ comprises a single binary feature $z^i \in \{ -1, +1 \}$ that impacts the label $y^i \in \{ -1, +1 \}$ via the logistic model
\begin{equation*}
    \mbox{Prob} (y \mid z) = \left[ 1 + \exp (-y \cdot \beta z) \right]^{-1}
\end{equation*}
with $\beta = 1$. \blue{(While we use $z^i \in \{0,1\}$ in the other sections of the paper, the analysis and results of this subsection and Appendix B assumed $z^i \in \{ -1, +1 \}$ as this made comparisons with \cite{NIPS2015} easier.)} We attempt to recover this logistic model in two ways:
\begin{enumerate}
    \item \emph{Mixed-feature model.} We employ our DR logistic regression~\eqref{eq:the_mother_of_all_problems} with a single categorical feature, using $p = 1$ and $\kappa = 1$ in the ground metric (\emph{cf.}~Definition~\ref{def:metric}).
    \item \emph{Continuous-feature model.} We employ a variant of problem~\eqref{eq:the_mother_of_all_problems} that treats the categorical feature $z_i$ as a continuous one. We choose as ground metric
    \begin{equation*}
        d (\bm{\xi}, \bm{\xi}') := \dfrac{1}{2}| z - z' | + \indicate{y \neq y'}.
    \end{equation*}
\end{enumerate}
We therefore use the feature distances $\indicate{z \neq z'}$ in the mixed-feature case and $\dfrac{1}{2} | z - z' |$ in the continuous-feature case. Note they both attain the value $0$ ($1$) for equal (different) feature values.
\begin{figure}[tb]
    \centering
    \includegraphics[width=.48\textwidth]{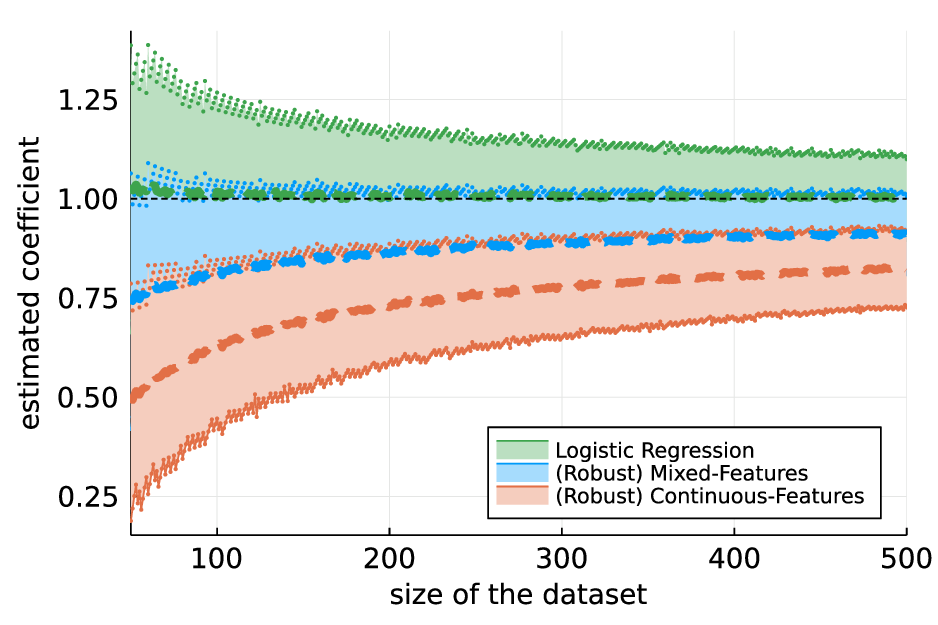} \hfill  \includegraphics[width=.48\textwidth]{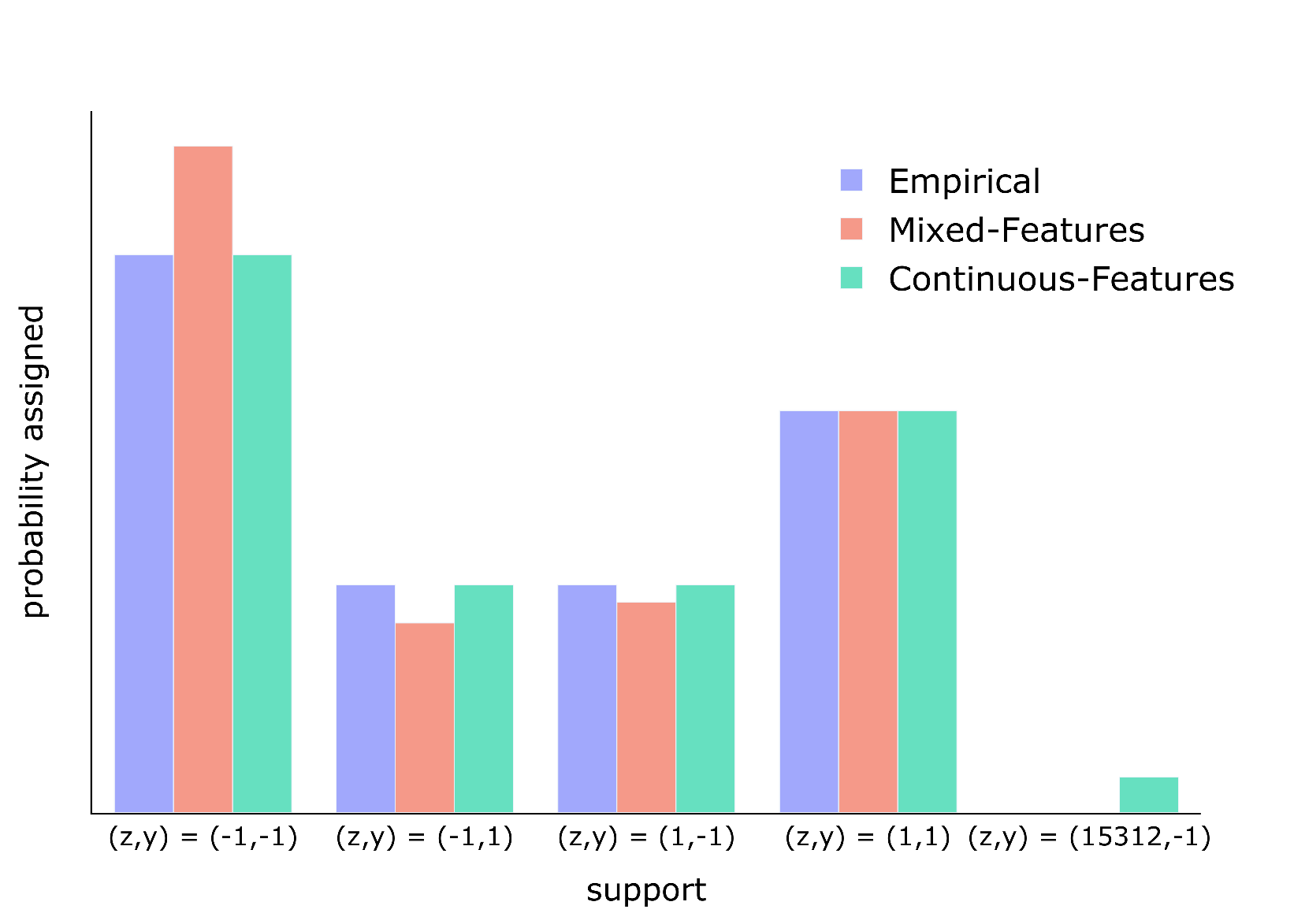}
    \caption{\emph{Left:} Estimates of $\beta$ for the standard logistic regression, our mixed-feature as well as the continuous-feature model as a function of the size $N$ of the data set. All results are reported as averages over 2,000 statistically independent runs. \emph{Right:} Comparison of the empirical distribution as well as the worst-case distributions for $\beta = 1$ under the two DR models for $N = 250$ samples. The probabilities are plotted on a log-scale to make small values visible.}
    \label{fig:were_simply_the_best}
\end{figure}

Figure~\ref{fig:were_simply_the_best}~(left) compares the mean values (dashed lines) as well as the 15\% and 85\% quantiles (shaded regions) of $\beta$ for the standard logistic regression, our mixed-feature as well as the continuous-feature model as a function of the size $N$ of the data set. For both DR models, we employed the same Wasserstein radius $\epsilon \propto 1 / \sqrt{N}$, which is motivated by the finite sample guarantee presented in the next section (\emph{cf.}~Theorem~\ref{thm:finite_sample_guarantee}). The figure shows that the continuous-feature model excessively shrinks its estimates of $\beta$, thus leading to overly conservative results. In fact, the true parameter value $\beta = 1$ consistently lies outside the confidence region of the continuous-feature model, independent of the sample size $N$. This is due to the fact that the continuous-feature model accounts for nonsensical worst-case distributions under which the categorical feature can take values outside its domain $\{ -1, +1 \}$, see Figure~\ref{fig:were_simply_the_best}~(right). In fact, the continuous-feature model places non-zero probability on an extreme scenario under which the binary feature $z \in \{-1, +1\}$ attains the value $15,312$. \blue{(In Appendix B we explain why this value of $15,312$ arises.)} Our mixed-feature model, on the other hand, restricts the worst-case distribution to the domain of the categorical feature and thus hedges against realistic distributions only.

We note that in practice, the radius $\epsilon$ of the Wasserstein ball will be chosen via cross-validation (\emph{cf.}~Section~\ref{sec:numericals}), in which case our mixed-feature model reliably outperforms the classical logistic regression on standard benchmark instances. While one may argue that the impact of nonsensical worst-case distributions in the continuous-feature model is alleviated by the cross-validated radii, Figure~\ref{fig:were_simply_the_best} offers strong theoretical and practical reasons against the use of a continuous-feature model for categorical or mixed-feature problems: The continuous-feature model hedges against a clumsily perturbed worst-case distribution that introduces a pronounced bias in the estimation. 
\blue{We also note that another possible approach would be to use the continuous-feature model but to restrict the support of binary features to $[-1,1]$. Unfortunately, for the log-loss function no tractable reformulation with support constraints appears to be known; see \cite{JMLR:v20:17-633} .}
Moreover, Section~\ref{sec:numericals} will show that the mixed-feature model, despite its exponential size, can be solved quickly and reliably with our novel column-and-constraint approach from Section~\ref{sec:cac_generation}.

\subsection{Complexity Analysis}\label{sec:mixed_feature_lr:complexity}

We first show that despite its exponential size, the DR logistic regression problem~\eqref{eq:the_mother_of_all_problems} admits a polynomial-time solution. Key to this perhaps surprising finding is the shape of the loss function: while problem~\eqref{eq:the_mother_of_all_problems} is strongly NP-hard (and thus unlikely to admit a polynomial-time solution scheme) for generic loss functions, it can be solved in polynomial time for the log-loss function $l_{\bm{\beta}}$ employed in logistic regression.

\begin{thm}[Complexity of the DR Logistic Regression~\eqref{eq:the_mother_of_all_problems}]\label{thm:complexity}
    ~
    \begin{enumerate}
        \item[\emph{(i)}] For generic loss functions $l_{\bm{\beta}}$, problem~\eqref{eq:the_mother_of_all_problems} is strongly NP-hard even if $n = 0$ and $N = 1$.
        \item[\emph{(ii)}] For the loss function $l_{\bm{\beta}} (\bm{x}, \bm{z}, y) = \log \left( 1 + \exp \left[ -y \cdot \left( \beta_0 + \bm{\beta}_{\text{\emph{N}}}{}^\top \bm{x} + \bm{\beta}_{\text{\emph{C}}}{}^\top \bm{z} \right) \right] \right)$ and the ground metric of Definition~\ref{def:metric}, problem~\eqref{eq:the_mother_of_all_problems} can be solved to $\delta$-accuracy in polynomial time.
    \end{enumerate}
\end{thm}

Recall that an optimization problem is solved to $\delta$-accuracy if a $\delta$-suboptimal solution is identified that satisfies all constraints modulo a violation of at most $\delta$. The consideration of $\delta$-accurate solutions is standard in the numerical solution of nonlinear programs where an optimal solution may be irrational.

A by now well-known result shows that when $m = 0$ (no categorical features), the DR logistic regression problem~\eqref{eq:the_mother_of_all_problems} reduces to a classical logistic regression with an additional regularization term $\lVert \bm{\beta}_{\bm{x}} \rVert_*$ in the objective function when the output label weight $\kappa$ in Definition~\ref{def:metric} approaches $\infty$ \cite{NIPS2015,JMLR:v20:17-633}. We next show that this reduction to a classical regularized logistic regression no longer holds in our problem setting when categorical features are present.

\begin{thm}[Absence of a Reformulation as a Regularized Problem]\label{thm:it_aint_regularized}
    Even when the output label weight $\kappa$ approaches $\infty$ in the DR logistic regression~\eqref{eq:the_mother_of_all_problems}, problem~\eqref{eq:the_mother_of_all_problems} does not admit an equivalent reformulation
    \begin{equation*}
        \begin{array}{l@{\quad}l}
            \displaystyle \mathop{\text{\emph{minimize}}}_{\bm{\beta}} & \displaystyle \mathbb{E}_{\color{black} \widehat{\mathbb{P}}_N} \left[ l_{\bm{\beta}} (\bm{x}, \bm{z}, y) \right] + \mathfrak{R} (\bm{\beta}) \\
            \displaystyle \text{\emph{subject to}} & \displaystyle \bm{\beta} = (\beta_0, \bm{\beta}_{\text{\emph{N}}}, \bm{\beta}_{\text{\emph{C}}}) \in \mathbb{R}^{1 + n + k},
        \end{array}
    \end{equation*}
    as a classical regularized logistic regression \emph{for any regularizer} $\mathfrak{R} : \mathbb{R}^{1 + n + k} \rightarrow \mathbb{R}$. 
\end{thm}

 Note that Theorem~\ref{thm:it_aint_regularized} does not only preclude the existence of a specific regularizer, but it excludes the existence of \emph{any} regularizer, no matter how complex its dependence on $\bm{\beta}$ might be. We are not aware of any prior results of this form in the literature. \blue{Some insight for the result in Theorem~\ref{thm:it_aint_regularized} may be found via the special case we consider in its proof in Appendix E and in Remark \ref{rem:NotReg} that follows it.}

\section{Column-and-Constraint Solution Scheme}\label{sec:cac_generation}

Our DR logistic regression problem~\eqref{eq:cone_head_reformulation} comprises exponentially many variables and constraints, which renders its solution as a monolithic exponential conic program challenging. Instead, Algorithm~\ref{alg:cac_gen} employs a column-and-constraint generation scheme, \blue{\textit{e.g.}, \cite{ZENG2013457}}, that alternates between \emph{(i)} the solution of relaxations of problem~\eqref{eq:cone_head_reformulation} that omit most of its variables and constraints and \emph{(ii)} adding those variables and constraints that promise to maximally tighten the relaxations.

\begin{algorithm}[tb]
    \begin{algorithmic}
        \caption{Column-and-Constraint Generation Scheme for Problem~\eqref{eq:cone_head_reformulation}.} \label{alg:cac_gen}
        \State Set $\text{LB}_0 = -\infty$ and $\text{UB}_0 = +\infty$
        \State Choose (possibly empty) subsets $\mathcal{W}^+, \mathcal{W}^- \subseteq [N] \times \mathbb{C}$
        \While{$\text{LB}_t \neq \text{UB}_t$}
            \State Let $\theta^\star$ be the optimal value and $(\bbeta, \lambda, \bm{s}, \bm{u}^\pm, \bm{v}^\pm)$ be a minimizer of problem~\eqref{eq:cone_head_reformulation} where:
             \begin{itemize}\setlength{\itemindent}{2em}
                \item the first constraint set is only enforced for all $(i, \bm{z}) \in \mathcal{W}^+$;
                \item the second constraint set is only enforced for all $(i, \bm{z}) \in \mathcal{W}^-$;
                \item we only include those $(u^+_{i,\bm{z}}, v^+_{i,\bm{z}})$ for which $(i, \bm{z}) \in \mathcal{W}^+$;
                \item we only include those $(u^-_{i,\bm{z}}, v^-_{i,\bm{z}})$ for which $(i, \bm{z}) \in \mathcal{W}^-$.
            \end{itemize}
            \State Identify a most violated index $(i^+, \bm{z}^+)$ of the first constraint set in~\eqref{eq:cone_head_reformulation} and add it to $\mathcal{W}^+$
            \State Identify a most violated index $(i^-, \bm{z}^-)$ of the second constraint set in~\eqref{eq:cone_head_reformulation} and add it to $\mathcal{W}^-$
            \State Let $\vartheta^+$ and $\vartheta^-$ be the violations of $(i^+, \bm{z}^+)$ and $(i^-, \bm{z}^-)$, respectively
            \State Update $\text{LB}_t = \theta^\star$ and $\text{UB}_t = \min \{ \text{UB}_{t-1}, \; \theta^\star + \log (1 + \max \{ \vartheta^+, \, \vartheta^- \}) \}$
            \State Update $t = t + 1$
        \EndWhile
    \end{algorithmic}
\end{algorithm}

\begin{thm}\label{thm:cac_scheme}
    Algorithm~\ref{alg:cac_gen} solves problem~\eqref{eq:cone_head_reformulation} in finitely many iterations. Moreover, $\textrm{\emph{LB}}_t$ and $\textrm{\emph{UB}}_t$ constitute monotone sequences of lower and upper bounds on the optimal value of problem~\eqref{eq:cone_head_reformulation}.
\end{thm}

A key step in Algorithm~\ref{alg:cac_gen} is the identification of most violated indices $(i, \bm{z}) \in [N] \times \mathbb{C}$ of the first and second constraint set in the reduced DR regression problem~\eqref{eq:cone_head_reformulation} for a fixed solution $(\bbeta, \lambda, \bm{s}, \bm{u}^\pm, \bm{v}^\pm)$. For the first constraint set in~\eqref{eq:cone_head_reformulation}, the identification of such indices requires for each data point $i \in [N]$ the solution of the combinatorial problem
\begin{equation} \label{eq:CombOp1}
    \mspace{-25mu}
    \begin{array}{l@{\quad}l@{\quad}l}
        \displaystyle \mathop{\text{maximize}}_{\bm{z}} & \displaystyle \min_{u^{+}, v^{+}} \left\{ u^{+} +  v^{+} \, : \, \left[ \begin{array}{l}
            \displaystyle (u^{+}, 1, -s_i - \lambda d_\text{C} (\bz,\bz^i)) \in \mathcal{K}_{\exp}  \\
            \displaystyle (v^{+}, 1, -y^i\bbeta_{\text{N}}{}^\top \bx^i -y^i\bbeta_{\text{C}}{}^\top \bz - y^i \beta_0 -s_i - \lambda d_\text{C} (\bz,\bz^i)) \in \mathcal{K}_{\exp}
        \end{array}
        \right] \right\} \\[4.5mm]
        \displaystyle \text{subject to} & \displaystyle \bm{z} \in \mathbb{C},
    \end{array}  
\end{equation}
where an optimal value greater than $1$ corresponds to a violated constraint; an analogous problem can be defined for the second constraint set in~\eqref{eq:cone_head_reformulation}. Despite its combinatorial nature, the above problem can be solved efficiently by means of Algorithm~\ref{alg:most_violated}.

\begin{algorithm}[tb]
    \begin{algorithmic}
        \caption{Identification of Most Violated Constraints in the Reduced Problem~\eqref{eq:cone_head_reformulation}.} \label{alg:most_violated}
        \For{$j \in \{ 1, \ldots, m \}$}
            \State Find a feature value $\bm{z}_j^\star$ that minimizes $y^i \cdot \bm{\beta}_{\text{C},j}{}^\top \bm{z}_j$ across all $\bm{z}_j \in \mathbb{C} (k_j) \setminus \{ \bm{z}^i_j \}$
        \EndFor
        \State Let $\pi : [m] \rightarrow [m]$ be an ordering such that
        \begin{equation*}
            y^i \cdot \bm{\beta}_{\text{C}, \pi(j)}{}^\top (\bm{z}^\star_{\pi (j)} - \bm{z}^i_{\pi (j)}) \leq y^i \cdot \bm{\beta}_{\text{C}, \pi(j')}{}^\top (\bm{z}^\star_{\pi (j')} - \bm{z}^i_{\pi (j')}) \qquad \forall 1 \leq j \leq j' \leq m
        \end{equation*}
        \State Set $\mathcal{W} = \emptyset$
        \For{$\delta \in \{ 0, 1, \ldots, m \}$}
            \State Update $\mathcal{W} = \mathcal{W} \cup \{ \bm{z} \}$, where $\bm{z}_j = \bm{z}^\star_j$ if $\pi (j) \leq \delta$ and $\bm{z}_j = \bm{z}^i_j$ otherwise
        \EndFor
        \State Determine the most violated constraint from the candidate set $\mathcal{W}$
    \end{algorithmic}
\end{algorithm}


\begin{thm}\label{thm:algo_complex}
    Algorithm~\ref{alg:most_violated} identifies a most violated constraint for a given data-point $i \in [N]$ in time $\mathcal{O} (k + n + m^2)$.
\end{thm}

Algorithm~\ref{alg:most_violated} first determines for each categorical feature $\bm{z}_j$, $j = 1, \ldots, m$, the feature value $\bm{z}^\star_j \in \mathbb{C} (k_j) \setminus \{ \bm{z}^i_j \}$ that, \emph{ceteris paribus}, contributes to a maximal constraint violation in time $\mathcal{O} (k)$. The subsequent step sorts the different features $j = 1, \ldots, m$ in descending order of their contribution to constraint violations in time $\mathcal{O} (m \log m)$. The last step of Algorithm~\ref{alg:most_violated}, finally, constructs the most violated constraint $\bm{z}$ across those in which $\delta = 0, \ldots, m$ categorical feature values deviate from $\bm{z}^i$ and subsequently picks the most violated constraint of these in time $\mathcal{O} (n + m^2)$. We note that the overall complexity of $\mathcal{O} (k + n + m^2)$ can be reduced to $\mathcal{O} (k + n + m \log m)$ by a clever use of data structures in the final step; for ease of exposition, we omit the details. Finally, since Algorithm~\ref{alg:most_violated} is applied to each data point $i \in [N]$, the overall complexity increases by a factor of $N$. \blue{We provide  additional intuition behind Algorithm~\ref{alg:most_violated} immediately before the proof of Theorem \ref{thm:algo_complex} in Appendix E.}


\section{Numerical Results}\label{sec:numericals}

Section~\ref{sec:num_res:runtimes} first compares the runtimes of our column-and-constraint scheme from Section~\ref{sec:cac_generation} with those of solving the DR logistic regression problem~\eqref{eq:cone_head_reformulation} na\"ively as a monolithic exponential conic program. We subsequently compare the classification performance of problem~\eqref{eq:the_mother_of_all_problems} with those of a classical unregularized and regularized logistic regression on standard benchmark instances with categorical features (Section~\ref{sec:num_res:cat_features}) and mixed features (Appendix C). 

All algorithms were implemented in Julia \cite{julia} (MIT license) and executed on Intel Xeon 2.66GHz processors with 8GB memory in single-core mode. We use MOSEK 9.3 \cite{mosek2010mosek} (commercial) to solve all exponential conic programs through JuMP \cite{JuMP} (MPL2 License). \blue{(We note the open source solvers Ipopt and CVXOpt could be used instead of MOSEK.)} All source codes and detailed results are available on GitHub (\url{https://github.com/selvi-aras/WassersteinLR}). 

\subsection{Runtime Comparison with Monolithic Formulation}\label{sec:num_res:runtimes}

\begin{figure}[tb]
    \centering
    \includegraphics[width=.32\textwidth]{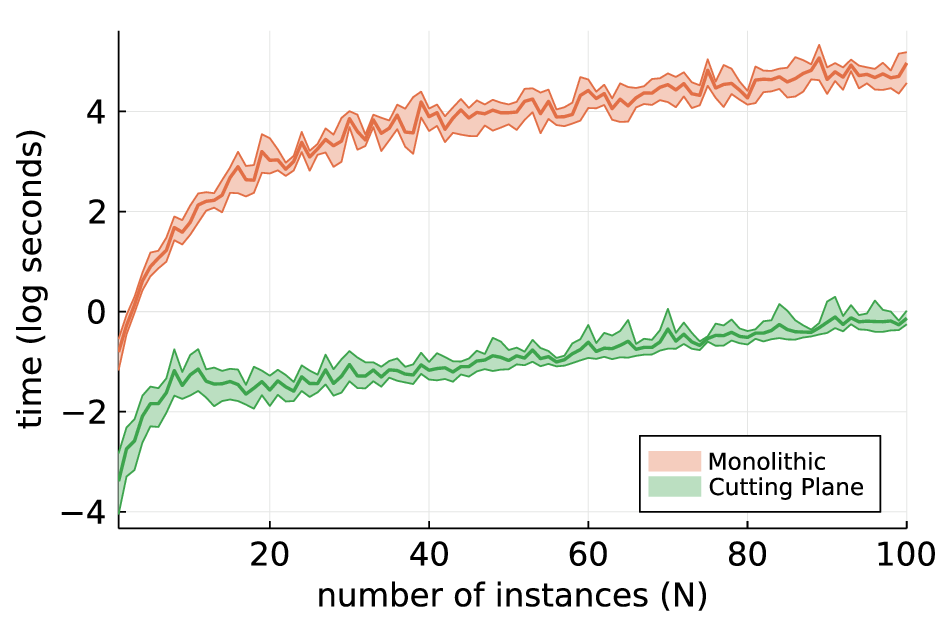} \hfill   \includegraphics[width=.32\textwidth]{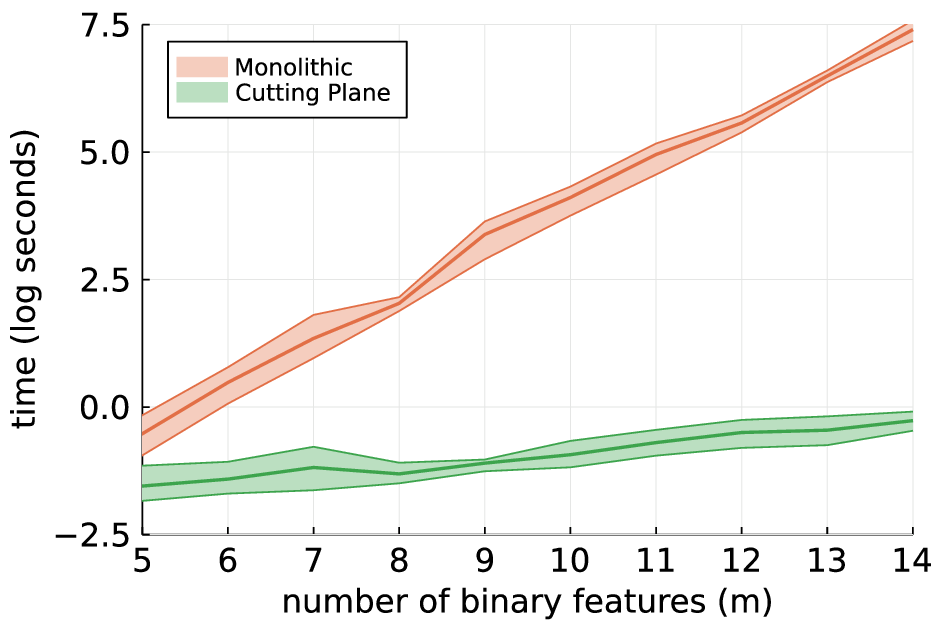} \hfill 
     \includegraphics[width=.32\textwidth]{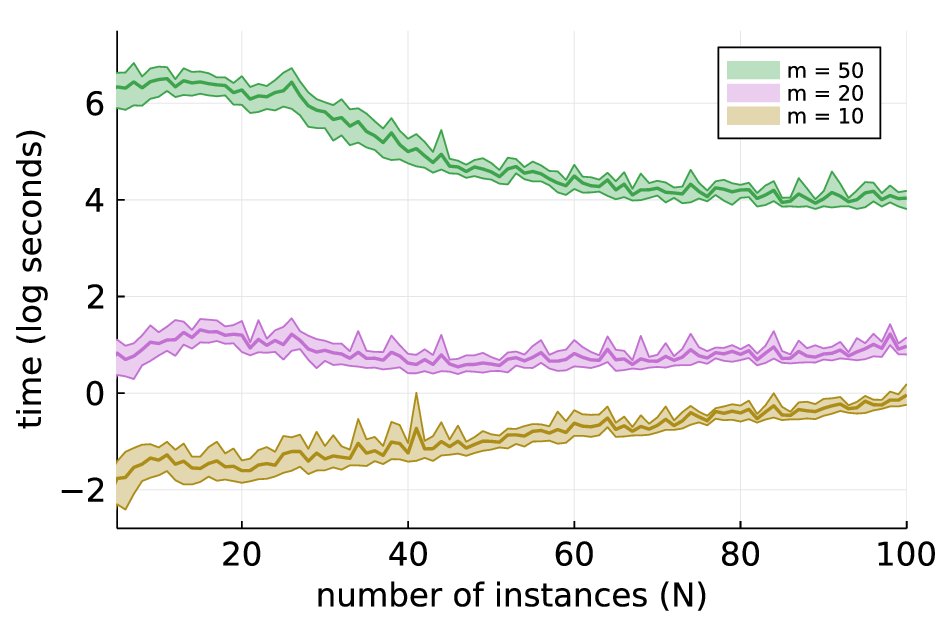}
    \caption{Runtime comparison between our column-and-constraint scheme and a na\"ive solution of problem~\eqref{eq:cone_head_reformulation} as a monolithic exponential conic program. \emph{Left:} Runtimes for $m = 10$ binary features as a function of the number $N$ of data points. \emph{Middle:} Runtimes for $N = 50$ as a function of $m$. \emph{Right:} Runtimes of our column-and-constraint scheme only for varying combinations of $m$ and $N$. In all graphs, shaded regions correspond to 10\%-90\% confidence regions and bold lines report median values over $50$ statistically independent runs. Note the log-scale in the {\color{black} plots}.} \label{fig:runtime_were_simply_the_best}
\end{figure}

We first compare the computation times of our column-and-constraint scheme from Section~\ref{sec:cac_generation} with those of solving the DR logistic regression problem~\eqref{eq:cone_head_reformulation} na\"ively as a monolithic exponential conic program. To this end, we randomly generate synthetic logistic regression instances with varying numbers $N$ of data points and $m$ of binary features. While Figure~\ref{fig:runtime_were_simply_the_best}~(left) shows that both approaches scale similarly in the number $N$ of data points, Figure~\ref{fig:runtime_were_simply_the_best}~(middle) reveals that the solution of the monolithic formulation scales exponentially in the number $m$ of binary features. In contrast, our column-and-constraint scheme scales gracefully in both the number $N$ of data points and the number $m$ of binary features (\emph{cf.}~Figure~\ref{fig:runtime_were_simply_the_best}, right). A similar behavior can be observed with more general, non-binary categorical features; we omit the results due to space constraints.

\subsection{Performance on Categorical-Feature Instances}\label{sec:num_res:cat_features}

We next compare the classification performance (in terms of the out-of-sample classification error) of our unregularized (`DRO') and Lasso-regularized (`r-DRO') DR logistic regression problem~\eqref{eq:the_mother_of_all_problems} with those of a classical unregularized (`LR') {\color{black} as well as Lasso-regularized (`r-LR'), mass transportation-regularized (MT) \cite{NIPS2015,JMLR:v20:17-633}, and robust Wasserstein profile inference-regularized (PI) \cite{blanchet_kang_murthy_2019}} logistic regression on the {\color{black} 14} most popular UCI data sets that only contain categorical features {\color{black} having more than 30 rows} \cite{UCI} (varying licenses). Instances with multiple output labels are indicated with a star; we convert them into instances with binary output labels by distinguishing between the majority class vs.~all other classes. The instances vary in the number $N$ of data points, the number $m$ of categorical features as well as, accordingly, the number $k$ of slopes considered in the one-hot encoding of the categorical features (\emph{cf.}~Section~\ref{sec:mixed_feature_lr:exponential_reform}). All results are reported as {\color{black} means} over {\color{black} 100} random training set-test set splits (80\%:20\%). The radius $\epsilon \in \{ 0, 10^{-5}, \ldots, 10^{-4}, \ldots, 1 \}$ of the Wasserstein ball as well as the Lasso penalty $\gamma \in \{ 0, \frac{1}{2} \cdot 10^{-5}, \ldots, \frac{1}{2} \cdot 10^{-4}, \ldots, \frac{1}{2} \}$ are selected via 5-fold cross-validation. We consider two variants of our DR logistic regression that employ a different output label weight ($\kappa = 1$ vs.~$\kappa = m$) in the ground metric (\emph{cf.}~Definition~\ref{def:metric}). The results are reported in Table~\ref{tab:categorical_results}. \blue{(Run-times and are provided in Appendix D where we also describe how statistical significance was assessed.)
The table shows that for the unregularized model, the classical logistic regression achieves the lowest classification error in 21\% of the instances, whereas our DR logistic regression achieve the lowest classification error in 36\% and 79\% of the instances for $\kappa = 1, \ m$, respectively. (The double-counting of ties means these percentages do not sum to 100\%.) For the regularized model, the results change to 14\% (classical logistic regression) vs.~57\% (each of our models). Within the numerical-feature Wasserstein DRO benchmarks, PI achieves the lowest classification error in 14\% of the instances, whereas MT achieves in 42\% and 57\% of the instances for $\kappa = 1, \ m$, respectively. Globally, one of our methods is the winning approach in 12 out of 14 datasets (in the remaining 2 they are the second best approach without statistically significant inferiority). In 9 of the datasets our methods win strictly (without a tie), and in 7 of these datasets the improvements are statistically significant over all other approaches.}

\begin{table}[tb]
{\color{black}
    \centering
    \resizebox{1.2\columnwidth}{!}{%
    \hskip-4.0cm
     \begin{tabular}{lccc|ccc|ccc|ccc}
\toprule
Data Set &  $N$ & $k$ & $m$ & LR  &  DRO ($\kappa=1$) &  DRO ($\kappa=m$)& r-LR  & r-DRO ($\kappa=1$) & r-DRO ($\kappa=m$) & MT ($\kappa = 1$) & MT ($\kappa = m$) & PI ($\alpha = 0.05$)\\
\midrule
     \underline{breast-cancer}  &  277  &  42 & 9 & 29.56\% & 29.15\% & \textbf{\textit{28.55\%}}\dag \ddag & 29.05\% & 29.16\% & \textit{28.82\%} & 29.44\% & 29.40\% & \textit{29.33\%} \\
     \underline{spect}  &267& 22 & 22  & 18.81\% & 17.72\% & \textit{17.51\%} & 17.79\% & 18.38\% & \textbf{\textit{15.83\%}}\dag \ddag & \textit{18.49\%} & 18.74\% & 20.60\% \\
     \underline{monks-3}  &554 & 11 & 6   & \textbf{\textit{2.07\%}} & 2.08\% & 2.10\% & 2.14\% & 2.19\% & \textit{2.13\%} & \textit{2.15\%} & \textit{2.15\%} & 37.87\%\\
      \underline{tic-tac-toe}  & 958 & 18 &  9   & 1.92\% & 1.69\% & \textbf{\textit{1.67\%}}\dag & \textit{1.68\%} & \textit{1.68\%} & 1.69\% & \textit{1.72\%} & 1.81\% & 30.02\% \\
      \underline{kr-vs-kp}  & 3,196 & 37 & 36 & \textbf{\textit{2.64\%}} & 2.66\% & \textbf{\textit{2.64\%}} & 2.69\% & \textit{2.66\%} & \textit{2.66\%} & 2.77\% & \textit{2.74\%} & 12.85\%\\
      \underline{balance-scale}$\star$ & 625 & 16 & 4  & 0.85\% & 0.82\% & \textit{0.72\%} & 0.73\% & 0.71\% &  \textbf{\textit{0.64\%}}\dag\ddag & 0.72\% & \textit{0.69\%} & 36.10\% \\
      \underline{hayes-roth}$\star$  & 160 & 11 & 4 & 17.00\% & 16.31\% & \textbf{\textit{16.09\%}}\dag \ddag & 18.66\% & 18.19\% & \textit{17.78\%} & 17.50\% & \textit{16.38\%} & 37.59\% \\
      \underline{lymphography}$\star$  & 148 & 42 &  18 & 21.17\% & 17.83\% & \textbf{\textit{16.72\%}}\dag \ddag & 17.38\% & 17.52\% & \textit{17.03\%} & 19.41\% & \textit{17.83\%} & 20.00\% \\
      \underline{car}$\star$  & 1,728 & 15 &  6  & \textbf{\textbf{\textit{4.73\%}}} & \textbf{\textbf{\textit{4.73\%}}} & \textbf{\textbf{\textit{4.73\%}}} & 4.99\% & \textit{4.76\%} & 4.87\% & \textbf{\textit{4.73\%}} & 4.77\% & 15.37\% \\
      \underline{splices}$\star$ & 3,189 & 229 & 60  & 6.80\% & \textbf{\textit{5.92\%}}\dag \ddag & 6.66\% & 6.22\% & \textit{5.99\%} & 6.02\% & \textit{6.57\%} & 6.87\% & 11.08\%  \\
      \underline{house-votes-84}  & 435  & 32 & 16  & 6.60\% & \textit{4.43\%} & 5.37\% & 4.54\% & \textit{4.41\%} & 5.16\% & \textit{4.78\%} & 5.54\% & \textbf{\textit{4.26\%}}\dag \\
      \underline{hiv}  & 6,590 & 152 & 8  & 5.94\% & 5.93\% & \textit{5.92\%} & \textbf{\textit{5.90\%}}\dag & \textbf{\textit{5.90\%}}\dag &\textbf{\textit{5.90\%}}\dag & 6.04\% & \textit{6.03\%} & 20.45\%  \\
      \underline{primacy-tumor}$\star$  & 339 & 25 & 17  & 13.93\% & \textbf{\textit{13.66\%}} & 13.76\% & 14.51\% & \textit{14.06\%} & 14.19\% & 13.87\% & \textit{13.85\%} & 19.18\% \\
      \underline{audiology}$\star$ & 226 & 92 & 69  & 14.18\% & \textit{14.03\%} & 14.40\% &3.11\%  & \textbf{\textit{2.80\%}}\dag\ddag & 3.91\% & 12.64\% & \textit{7.49\%} & 15.49\%\\
\bottomrule
\end{tabular}} \medskip
    \caption{\color{black} Classification errors of unregularized as well as Lasso-, mass transportation-, and profile inference-regularized variants of the classical logistic regression and our DR regression on UCI benchmark instances with categorical features only. The smallest error within each model group (unregularized vs.~regularized) is highlighted in italics, whereas the smallest error overall (across all groups) is printed in bold. The dagger (\dag) and double dagger (\ddag) symbols next to the best model denote statistically significant improvements over LR and the second best model, respectively.}
    \label{tab:categorical_results} \vspace{-0.5cm}
}
\end{table}

\section{Conclusions}

We proposed a new DR mixed-feature logistic regression model where the proximity between the empirical distribution and the unknown true data-generating distribution is measured by the popular Wasserstein distance. Despite its exponential-size formulation, we prove that the underlying optimization problem can be solved in polynomial time, and we develop a practically efficient column-and-constraint generation scheme for its solution. The promising performance of our model is demonstrated in numerical experiments on standard benchmark instances. We note that our column-and-constraint scheme readily extends to other DR mixed-feature machine learning models such as linear regression, support vector machines and decision trees. Applying our algorithm in those settings is a promising direction for future research. \blue{A further interesting direction is the development of tailored solution schemes rather than using off-the-shelf solvers for solving our DR mixed-feature logistic regression problem.}


\clearpage
\newpage
\begin{ack}
\blue{Funding from the EPSRC grant EP/W003317/1 is gratefully acknowledged. The first author acknowledges support from The Alan Turing Institute. 
The authors are grateful for the comments and suggestions of the anonymous reviewers, which have helped to significantly improve the manuscript. This work has been produced without the involvement of any commercial entities that may cause a conflict of interest or competing interests.}
\end{ack}
\bibliographystyle{plainnat}
\bibliography{bibliography}


\newpage

\section*{Appendix A: Performance Guarantees}\label{app:mixed_feature_lr:performance}

Wasserstein ambiguity sets benefit from measure concentration results that characterize the rate at which the empirical distribution $\widehat{\mathbb{P}}_N$ converges to the unknown true distribution $\mathbb{P}^0$. In the following, we review existing results from the literature to characterize the finite sample and asymptotic guarantees of our DR logistic regression~\eqref{eq:the_mother_of_all_problems}.\\

\begin{thm}[Finite Sample Guarantee]\label{thm:finite_sample_guarantee}
    Assume that $\mathbb{P}^0$ is light-tailed, that is, $\mathbb{E}_{\mathbb{P}^0} \left[ \exp (\lVert \bm{\xi} \rVert^a) \right] \leq A$ for some $a > 1$ and $A > 0$. Then there are $c_1, c_2 > 0$ only depending on $\mathbb{P}^0$ through the light-tail parameters $a$, $A$ and the feature space dimensions $(n, m)$ such that any optimizer $\bm{\beta}^\star$ to~\eqref{eq:the_mother_of_all_problems} satisfies
    \begin{equation*}
        [\mathbb{P}^0]^N \left( \mathbb{E}_{\mathbb{P}^0} \left[ l_{\bm{\beta}^\star} (\bm{x}, \bm{z}, y) \right] \, \leq \, \sup_{\mQ \in \mathfrak{B}_{\epsilon}(\widehat{\mP}_N)} \ \mE_{\mQ} \left[ l_{\bm{\beta}^\star}(\bx,\bz,y) \right] \right) \; \geq \; 1 - \eta
    \end{equation*}
    for any confidence level $\eta \in (0, 1)$ and Wasserstein ball radius
    \begin{equation*}
        \mspace{-10mu}
        \epsilon \; \geq \;
        \left( \frac{\log (c_1 / \eta)}{c_2 N} \right)^{1 / \max \{ m+n+1, \, 2 \}} \cdot \mathds{1} \left[ N \geq \frac{\log(c_1 / \eta)}{c_2} \right]
        \; + \;
        \left( \frac{\log (c_1 / \eta)}{c_2 N} \right)^{1 / \alpha} \cdot \mathds{1} \left[ N < \frac{\log(c_1 / \eta)}{c_2} \right].
    \end{equation*}
\end{thm}\medskip

Recall that $[\mathbb{P}^0]^N$ in the statement of Theorem~\ref{thm:finite_sample_guarantee} refers to the $N$-fold product distribution of $\mathbb{P}^0$ that governs the data set $\{ \bm{\xi}^i \}_{[i \in N]}$ upon which the optimizer(s) $\bm{\beta}^\star$ of problem~\eqref{eq:the_mother_of_all_problems} depend(s) via $\widehat{\mP}_N$. Theorem~\ref{thm:finite_sample_guarantee} shows that with arbitrarily high probability $1 - \eta$, the optimal value $\sup_{\mQ \in \mathfrak{B}_{\epsilon}(\widehat{\mP}_N)} \ \mE_{\mQ} \left[ l_{\bm{\beta}^\star}(\bx,\bz,y) \right]$ of our DR logistic regression~\eqref{eq:the_mother_of_all_problems} \emph{over}estimates the loss $\mathbb{E}_{\mathbb{P}^0} \left[ l_{\bm{\beta}^\star} (\bm{x}, \bm{z}, y) \right]$ incurred by any optimal solution $\bm{\beta}^\star$ under the unknown true distribution $\mathbb{P}_0$ as long as the radius $\epsilon$ of the Wasserstein ball $\mathfrak{B}_{\epsilon}(\widehat{\mP}_N)$ is sufficiently large. Since the categorical features attain finitely many different values, the bound of Theorem~\ref{thm:finite_sample_guarantee} can be sharpened by replacing $m + n + 1$ with $n + 1$ if the constants $a$ and $A$ are adapted accordingly. We emphasize that the decay rate of $\mathcal{O} (N^{-1/(n+1)})$ in Theorem~\ref{thm:finite_sample_guarantee} is essentially optimal; see \cite[\S 3]{kmns19}.

To study the asymptotic consistency of problem~\eqref{eq:the_mother_of_all_problems} as well as the existence of sparse worst-case distributions, we first introduce a technical assumption.

\begin{defi}[Growth Condition]\label{defi:growth_condition}
    We say that the DR logistic regression~\eqref{eq:the_mother_of_all_problems} satisfies the \emph{growth condition} if \emph{(i)} the hypotheses $\bm{\beta}$ are restricted to a bounded set $\mathcal{H} \subseteq \mathbb{R}^{1 + n + k}$; and \emph{(ii)} there is $\bm{\xi}^0 \in \Xi$ and $C > 0$ such that $l_{\bm{\beta}} (\bm{\xi}) \leq C [1 + d (\bm{\xi}, \bm{\xi}^0)]$ across all $\bm{\beta} \in \mathcal{H}$ and $\bm{\xi} \in \Xi$.
\end{defi}

\begin{lem}\label{lem:growth_condition}
    If we restrict the hypotheses $\bm{\beta}$ to a bounded set $\mathcal{H} \subseteq \mathbb{R}^{1 + n + k}$, then the DR logistic regression~\eqref{eq:the_mother_of_all_problems} satisfies the growth condition of Definition~\ref{defi:growth_condition}.
\end{lem}

We are now in the position to study the asymptotic consistency of problem~\eqref{eq:the_mother_of_all_problems}.

\begin{thm}[Asymptotic Consistency]\label{thm:asymptotic_consistency}
    Under the assumptions of Theorem~\ref{thm:finite_sample_guarantee}, we have
    \begin{equation*}
        \sup_{\mQ \in \mB_{\epsilon_N}(\widehat{\mP}_N)} \ \mE_{\mQ} \left[ l_{\bm{\beta}^\star}(\bx,\bz,y) \right]
        \; \underset{N \rightarrow \infty}{\longrightarrow} \;
        \mathbb{E}_{\mathbb{P}^0} \left[ l_{\bm{\beta}^\star} (\bm{x}, \bm{z}, y) \right]
        \quad \mathbb{P}^0\text{-a.s.}
    \end{equation*}
    whenever $(\eta_N, \epsilon_N)$ is set according to Theorem~\ref{thm:finite_sample_guarantee} for all $N \in \mathbb{N}$, $\sum_N \eta_N < \infty$, $\lim_{N \rightarrow \infty} \epsilon_N = 0$, and the growth condition in Definition~\ref{defi:growth_condition} is satisfied.
\end{thm}

Theorem~\ref{thm:asymptotic_consistency} shows that the DR logistic regression~\eqref{eq:the_mother_of_all_problems} achieves asymptotic consistency if the \mbox{(un-)confidence} parameter $\eta$ and the radius $\epsilon$ of the Wasserstein ball are reduced simultaneously. Thus, any optimal solution to~\eqref{eq:the_mother_of_all_problems} converges to the optimal solution of the (non-robust) logistic regression under the unknown true distribution $\mathbb{P}^0$ when the size of the data set increases.

The proof of Theorem~\ref{thm:coneheads} shows that the optimization problem characterizing the worst-case distribution $\mathbb{Q}^\star \in \mathfrak{B}_{\epsilon}(\widehat{\mP}_N)$ comprises exponentially many decision variables. It is therefore natural to investigate the complexity of worst-case distributions to our  DR logistic regression~\eqref{eq:the_mother_of_all_problems}. The next result shows that there exist worst-case distributions that exhibit a desirable sparsity pattern: their numbers of atoms scale with the number of data samples.

\begin{thm}[Existence of Sparse Worst-Case Distributions]\label{thm:existence_sparse_distributions}
    Assume that the growth condition in Definition~\ref{defi:growth_condition} is satisfied. Then there are worst-case distributions $\mathbb{Q}^\star \in \mathfrak{B}_{\epsilon}(\widehat{\mP}_N)$ satisfying
    \begin{equation*}
        \mE_{\mQ^\star} \left[ l_{\bm{\beta}^\star}(\bx,\bz,y) \right]
        \; = \;
        \sup_{\mQ \in \mathfrak{B}_{\epsilon}(\widehat{\mP}_N)} \ \mE_{\mQ} \left[ l_{\bm{\beta}^\star}(\bx,\bz,y) \right]
    \end{equation*}
    such that $\mathbb{Q}^\star$ is supported on at most $N + 1$ atoms.
\end{thm}

\blue{
Our performance guarantees in this section scale with the dimension of the feature space as we seek for a high confidence of the unknown true distribution being contained in our ambiguity set $\mathfrak{B}_\epsilon (\widehat{\mathbb{P}}_N)$. A dimension-independent performance guarantee can be obtained along the lines of \cite{blanchet_kang_murthy_2019} if one instead only seeks for a high confidence of the unknown true model $\bm{\beta}^\star$ being contained in the union of optimal classifiers corresponding to the individual distributions $\mathbb{P}$ contained in $\mathfrak{B}_\epsilon (\widehat{\mathbb{P}}_N)$. We refer to \cite{https://doi.org/10.1002/wics.1587} for a detailed review of the achievable performance guarantees of Wasserstein ambiguity sets with respect to the dimension of the feature space.}

{\color{black}
\section*{Appendix B: Finding the Worst Distribution in Example of Section~\ref{sec:comparison_to_continuous}}
We provide further details here on the derivation of the worst distribution associated with the benchmark in Section~\ref{sec:comparison_to_continuous} where we treat the (only) binary feature as a numerical feature and use the DR logistic regression algorithm of \cite{NIPS2015}. We have a single binary feature $z \in \{-1,1\}$ and there is no intercept term so the log-loss function is $l_{\beta}(z,y) = \log(1 + \exp(-\beta\cdot y\cdot z))$. The true unknown value of $\beta$ is $\beta = 1$, we have a data set $(z^i, y^i)_{i \in [N]}$, and we take $p=1$, $\kappa=1$ (distance metric parameters), and $\epsilon =  1 / (2\sqrt{N})$. We would like to solve the worst distribution problem
\begin{align}\label{eq:worst_dist}
    \displaystyle \sup_{\mathbb{Q} \in \mathfrak{B}_\epsilon (\widehat{\mathbb{P}}_N)} \; \mathbb{E}_\mathbb{Q} \left[ l_{{\beta}} (z, y) \right].
\end{align}

We can do this by considering  the following problem taken from \cite[Thm 20]{JMLR:v20:17-633} and adopted for our specific setting:
\begin{equation}\label{solve_this_for_prize}
    \begin{array}{l@{\quad}l@{\qquad}l}
        \displaystyle \mathop{\text{maximize}}_{\theta,\; \{ \alpha_i\}_{i \in [N]}} & \displaystyle \theta + \dfrac{1}{N} \sum_{i =1}^N \left[ (1- \alpha_i) l_\beta(z^i, y^i) + \alpha_i l_{\beta}(z^i, -y^i) \right] \\[4mm]
        \displaystyle \text{subject to} & \displaystyle \theta + \dfrac{1}{N} \sum_{i=1}^N  \alpha_i = \epsilon - \gamma \\[4mm]
        & \displaystyle 0 \leq \alpha_i \leq 1, \quad i \in [N] \\[3mm]
        & \displaystyle \theta \geq 0. 
    \end{array}
\end{equation}
It is parametrized by some $\gamma \in [0, \min\{ \epsilon, 1\}]$ and \cite[Thm 20]{JMLR:v20:17-633} show that its optimal value  for the case $\gamma = 0$ coincides with the optimal value of problem~\eqref{eq:worst_dist}. Furthermore, if we denote by $(\theta^\star(\gamma), \ \{\alpha^\star_i(\gamma) \}_{i \in [N]})$ the optimal solutions to (\ref{solve_this_for_prize}), then the sequence of probability distributions 
\begin{align}\label{eq:distribution_to_derive}
\begin{array}{ll}
    \displaystyle \mathbb{Q}_{\gamma} = &  \displaystyle  \dfrac{1}{N} \sum_{i=2}^N \left[ (1- \alpha^\star_i(\gamma)) \delta_{(z^i, y^i)} + \alpha^\star_i(\gamma)\delta_{(z^i, -y^i)} \right] + \dfrac{\eta(\gamma)}{N} \delta_{(z^1 + \frac{\theta^\star(\gamma)N}{\eta(\gamma)}, y^1)} \\[4mm]
    & \displaystyle  + \dfrac{1 - \eta(\gamma)}{N} \left[ (1 - \alpha^\star_1(\gamma))\delta_{(z^1, y^1)} + \alpha^\star_1(\gamma) \delta_{(z^1, -y^1)} \right]
\end{array}
\end{align}
where $\eta(\gamma) := \gamma / (\theta^\star(\gamma) + 2 - \epsilon + \gamma)$, constructs an asymptotically optimal solution to problem~\eqref{eq:worst_dist} as $\gamma \downarrow 0$. In order to explicitly characterize this sequence, we derive a closed-form solution to problem~\eqref{solve_this_for_prize}. Using the equality constraint to substitute for $\theta$ in the objective in (\ref{eq:distribution_to_derive}) yields
\begin{equation*}
    \begin{array}{l@{\quad}l@{\qquad}l}
        \displaystyle \mathop{\text{maximize}}_{\{ \alpha_i\}_{i \in [N]}} & \displaystyle \epsilon - \gamma - \dfrac{1}{N} \sum_{i=1}^N  \alpha_i + \dfrac{1}{N} \sum_{i =1}^N \left[ (1- \alpha_i) l_\beta(z^i, y^i) + \alpha_i l_{\beta}(z^i, -y^i) \right] \\[4mm]
        \displaystyle \text{subject to} & \displaystyle 0 \leq \alpha_i \leq 1, \quad i \in [N]\\[3mm]
        &\displaystyle  \epsilon - \gamma \geq  \dfrac{1}{N}\sum_{i=1}^N \alpha_i.
    \end{array}
\end{equation*}
Ignoring the constant terms in the objective and re-arranging terms yields
\begin{equation*}
    \begin{array}{l@{\quad}l@{\qquad}l}
        \displaystyle \mathop{\text{maximize}}_{\{ \alpha_i\}_{i \in [N]}} & \displaystyle \sum_{i =1}^N \alpha_i \left[ -1 + l_\beta(z^i, -y^i) - l_\beta(z^i, y^i)  \right] \\[4mm]
        \displaystyle \text{subject to} & \displaystyle 0 \leq \alpha_i \leq 1, \quad i \in [N]\\[2mm]
        &\displaystyle  \epsilon - \gamma \geq  \dfrac{1}{N}\sum_{i=1}^N \alpha_i.
    \end{array}
\end{equation*}
Since $z^i \in \{-1, +1\}, \ i \in [N]$ holds, we have $l_\beta(z^i, -y^i) - l_\beta(z^i, y^i) \in \{-1, +1\}$, which implies that the coefficients of $\alpha_i$ in the objective function are non-positive. Hence, $\alpha^\star_i = 0, \ i \in [N]$ is an optimal solution to the  problem. This implies $\alpha^\star_i(\gamma) = 0, \ i \in [N]$ and $\theta^\star(\gamma) = \epsilon - \gamma$ are optimal in problem~\eqref{solve_this_for_prize}.

We can then observe from~\eqref{eq:distribution_to_derive} that the worst distribution places $1/N$ mass on each data point $i = 2, \ldots, N$, and the remaining $1 / N$ mass is distributed as: $(1 - \eta(\gamma))/N$ mass on the data point $i = 1$ and $\eta(\gamma)/N$ mass on the point $(z^1 + \frac{(\epsilon - \gamma)N}{\gamma / 2}, y^1)$ which is \textbf{not} in the data-set and in fact is an infeasible point. In our experiments, we take $\gamma = 10^{-3}$, because the optimal value of problem~\eqref{solve_this_for_prize} numerically converges after this value, and one can then verify that in the setting we work with ($N = 250)$ a mass of $\frac{10^{-3}}{500}$ is placed on the point with feature $z^1 + \frac{(\epsilon - \gamma)N}{\gamma / 2} = 1 + \frac{(\frac{1}{1\sqrt{N}} - 10^{-3})N}{10^{-3} / 2} \approxeq 15,312$ and label $y^1$. This summarizes the specific approach we took to obtain the worst distribution in Figure~\ref{fig:were_simply_the_best} and demonstrates the key problem that arises with treating categorical variables as numerical.
} 

\section*{Appendix C: Numerical Results on UCI Data-Sets with Mixed Features} \label{sec:num_res:mixed_features}

We repeat the experiment of the Section \ref{sec:num_res:cat_features} for the five most popular mixed-feature instances of the UCI data set \cite{UCI}. The results are reported in Table~\ref{tab:mixed_results}. \blue{All results are reported as medians over 20 random training set-test set splits ($80\%$:$20\%$). Cross-validation is applied for the same set of parameters as Section~\ref{sec:num_res:cat_features} over identical grids.} The conclusions are qualitatively similar to those of  Section~\ref{sec:num_res:cat_features}. We highlight, however, that now only two of the lowest classification errors are achieved by one of the non-robust models, whereas the lowest classification error in each instance is obtained by at least one of the robust models we propose. 

\begin{table}[!h]
{\color{black}
    \centering
    \resizebox{1.2\columnwidth}{!}{%
    \hskip-4.0cm
     \begin{tabular}{lcccc|ccc|ccc|ccc}
\toprule
Data Set &  $N$ & $n$ & $k$ & $m$ & LR  &  DRO ($\kappa=1$) &  DRO ($\kappa=m$)& r-LR  & r-DRO ($\kappa=1$) & r-DRO ($\kappa=m$) & MT ($\kappa = 1$) & MT ($\kappa = m$) & PI ($\alpha = 0.05$)\\
\midrule
     \underline{credit-approval}  &  690  & 6 & 36 & 9 & 14.13\% & 13.77\% & \textbf{\textit{13.04\%}}\dag\ddag & 14.13\% & 14.13\% & \textbf{\textit{13.04\%}}\dag\ddag & 14.13\% & 14.13\% & \textit{13.41\%} \\
      \underline{annealing}$\star$ & 798 & 6 & 46 & 32 & \textit{2.52\%} &  2.83\% & \textit{2.52\%} & 2.20\% & \textbf{\textit{1.89\%}}\dag\ddag & \textbf{\textit{1.89\%}}\dag\ddag &  2.52\% & \textit{2.20\%} & 12.58\% \\
      \underline{contraceptive}$\star$ & 1,473 & 2 & 15 & 7 & \textit{33.16\%} &  \textit{33.16\%} & \textit{33.16\%} & \textbf{\textit{32.82\%}}\dag & 32.99\% & \textbf{\textit{32.82\%}}\dag & \textit{33.16\%} &\textit{ 33.16\%} & 39.63\% \\
      \underline{hepatite} & 155 & 6 & 23 & 13 & \textbf{\textit{16.13\%}} & \textbf{\textit{16.13\%}} & 19.35\% & 19.35\% & 17.74\% & \textbf{\textit{16.13\%}} &  \textbf{\textit{16.13\%}} &  \textbf{\textit{16.13\%}} &  17.74\%\\
      \underline{cylinder-bands} & 539 & 19 & 43 & 15 & 23.36\% & \textbf{\textit{21.50\%}}\dag\ddag & \textbf{\textit{21.50\%}}\dag\ddag & 23.36\% & \textbf{\textit{21.50\%}}\dag\ddag & 22.43\% & \textit{22.43\%} & \textit{22.43\%} & 30.84\%\\
\bottomrule
\end{tabular}} \medskip
    \caption{\color{black} Classification errors of unregularized and Lasso-regularized variants of the classical logistic regression and our DR regression on UCI benchmark instances with mixed features. We use the same notation and highlighting conventions as in Table~\ref{tab:categorical_results}.}
    \label{tab:mixed_results} \vspace{-0.5cm}
}
\end{table}

\begin{table}[!h]
{\color{black}
    \centering
    \resizebox{1.2\columnwidth}{!}{%
    \hskip-4.0cm
     \begin{tabular}{lcccc|ccc|ccc|ccc}
\toprule
Data Set &  $N$ & $n$ & $k$ & $m$ & LR  &  DRO ($\kappa=1$) &  DRO ($\kappa=m$)& r-LR  & r-DRO ($\kappa=1$) & r-DRO ($\kappa=m$) & MT ($\kappa = 1$) & MT ($\kappa = m$) & PI ($\alpha = 0.05$)\\
\midrule
     \underline{credit-approval}  &  690  & 6 & 36 & 9 & 0.11 & 53.46 & 55.89 & 0.09 & 36.77 & 43.58 & 0.36 & 0.56 & 0.07 + 0.08 \\
      \underline{annealing}$\star$ & 798 & 6 & 46 & 32 & 0.18 & 62.96 & 71.09 & 43.58 & 22.91 & 0.24 &  0.36 & 0.58 & 0.10 + 0.12 \\
      \underline{contraceptive}$\star$ & 1,473 & 2 & 15 & 7 & 0.14 &  86.06 & 81.70 & 0.15 & 91.10 & 84.04 & 0.52 & 0.74 & 0.08 + 0.03\\
      \underline{hepatite} & 155 & 6 & 23 & 13 & 0.05 & 9.01 & 10.59 & 0.02 & 2.49 & 4.86 &  0.06 &  0.10 &  0.01 + 0.05\\
      \underline{cylinder-bands} & 539 & 19 & 43 & 15 & 0.13 & 96.10 & 106.48 & 0.09 & 67.09 &  41.63 & 0.32 & 0.59 & 0.06 + 0.12\\
\bottomrule
\end{tabular}} \medskip
    \caption{\color{black} Median runtimes (in seconds) associated with Table~\ref{tab:mixed_results}.}
    \label{tab:mixed_results_runtimes} \vspace{-0.5cm}
}
\end{table}

\section*{Appendix D: Further Details on Numerical Experiments}

Throughout the numerical experiments, we fixed $p = 1$ in the ground metric (\emph{cf.}~Definition~\ref{def:metric}).

\textbf{Synthetic data sets:} In Section~\ref{sec:num_res:runtimes} we generate synthetic data sets with $N$ data points and $m$ binary features. The data generation process is summarized in Algorithm~\ref{alg:synthetic}.

\begin{algorithm}[!htb]
    \begin{algorithmic}
        \caption{Construction of synthetic data sets in Section~\ref{sec:num_res:runtimes}.}
        \State Sample the components of $\beta_0$ and $\bbeta_{\text{C}}$ i.i.d.\@ from a standard normal distribution 
        \State Normalize $\beta_0$ and $\bbeta_{\text{C}}$ by dividing them with $\lVert \bm{\beta} \rVert_2$, where $\bm{\beta} = (\beta_0, \bbeta_{\text{C}})$
        \For{$i \in \{ 1, \ldots, N \}$}
            \State Construct $\bm{z}^i \in \mathbb{C}(2) \times \ldots \times \mathbb{C}(2)$ by sampling $\bm{z}^i \sim \{0, 1\}^m$ uniformly at random
            \State Find $p^i$, the probability of the $i$th data point having label $+1$, by using the `true' $\beta_0$ and $\bbeta_{\text{C}}$:
            $$ 
    p^i = \left[ 1 + \exp (- [\beta_0 + \bbeta_{\text{C}}{}^\top \bz^i]) \right]^{-1}$$
            \State Sample $y^i \in \{ -1, 1 \}$ from a Bernoulli distribution with parameter $p^i$
        \EndFor
        \State The synthetic data set is the collection of data points $\{ (\bz^i, y^i) \}_{i \in [N]}$ constructed above
        \label{alg:synthetic}
    \end{algorithmic}
\end{algorithm}

\textbf{UCI data sets:} The datasets of Sections~\ref{sec:num_res:cat_features} and
Appendix C were taken from the UCI repository \cite{UCI}. Missing values (NaNs) were encoded as a new category of the corresponding feature; the dataset \underline{breast-cancer} {\color{black} is an exception}, where rows with missing values were dropped. Some data sets include features that are derivatives of the labels, primary keys of the instances (\emph{e.g.}, the \underline{auidology} data set), or they have features with only one possible category (\emph{e.g.}, the \underline{annealing} and \underline{cylinder-bands} data sets). We removed such features manually. We also removed rows that were readily identified as erroneous (\emph{e.g.}, two rows in \underline{cylinder-bands} had less columns than required). Data sets with multi-class (\emph{i.e.}, non-binary) labels were converted to binary labels by distinguishing between the majority label class and all other classes. If a data set includes separate training and test sets, we merged them and subsequently applied our training set-test set split as described in the main paper. As per standard practice, we randomly permuted the rows of each data set before conducting the splits into training and test sets. For the detailed data processing steps, we refer to the GitHub repository accompanying this paper, which also contains a sample Python 3 script.\footnote{\url{https://github.com/selvi-aras/WassersteinLR}}

\textbf{Implementation of the column-and-constraint generation scheme:} In our implementation, we switch between solving the primal and dual exponential conic reformulation of the DR logistic regression problem, as we found that sometimes the dual problem can be solved more easily than the primal. We have further implemented an `easing' step in the column-and-constraint generation scheme that periodically deletes constraints from $\mathcal{W}^+$ and $\mathcal{W}^-$ whose slacks exceed a pre-specified threshold. To this end, note that the constraints $u^+_{i, \bz} + v^+_{i,\bm{z}} \leq 1$, $(i, \bm{z}) \in \mathcal{W}^+$, in the relaxations of the exponential conic reformulation~\eqref{eq:cone_head_reformulation} have a slack ranging between $0$ and $1$ by construction. We implemented variants of our column-and-constraint generation scheme that conduct easing steps either every $200$ iterations or in iteration $t = 100 \cdot 1.5^k$, $k \in \mathbb{N}$. We have also implemented variants that keep the slack threshold constant at $0.05$ and where this threshold starts at $0.02$ and is subsequently increased by $0.02$ in each easing step. In all of our variants, a constraint is deleted at most once, that is, it is no longer considered for deletion if it has been reintroduced after a prior deletion. This ensures that Algorithm~\ref{alg:cac_gen} terminates in finite time without cycling. Analogous steps have been implemented for the constraints $u^-_{i, \bz} + v^-_{i,\bm{z}} \leq 1$, $(i, \bm{z}) \in \mathcal{W}^-$.


{\color{black}
\textbf{Determining statistical Significance:} In the numerical experiments of Section \ref{sec:num_res:cat_features} we compare the means of $100$ out-of-sample errors attained by several logistic regression methods and also identify which methods appeared to be statistically significant. In Table~\ref{tab:categorical_results}, for example, a dagger (\dag) symbol next to the winning approach denotes statistically significant error improvement over the standard logistic regression (`LR') and a double dagger (\ddag) symbol denotes such improvement over the second best approach. Note that if the winning approach is a variant of the DRO methods we propose, then the second best approach is taken over the methods excluding our methods for a fair comparison. 

We add a dagger (\dag) to the winning approach according to the following approach. Firstly, we  subtract (element-wise) the vector of errors attained by standard logistic regression from the vector of errors attained by the winning approach. Each element of the new vector is the `additional error' the winning approach has compared to the standard logistic regression. We then try to reject the hypothesis (at $5\%$-significance level) that this additional error is non-negative (\textit{i.e.}, we try to reject the hypothesis that the standard logistic regression is at least as good). To this end, we compute the t-statistic for the mean of the additional errors vector with a hypothesis mean of $0$ and sample size of $100$. We then compute the cumulative probability of this value via a one-sided t-test (with $100 - 1$ degrees of freedom) to obtain a p-value. If this value is less than $0.05$, then we reject the hypothesis, concluding that the improvement is significant. Note that, with this approach we implicitly assume the out-of-sample errors are independent, which is typically not the case \cite{salzberg1997comparing}. Hence, we acknowledge that these tests of significance are only approximate. The presence of a double dagger is determined analogously.
}

\textbf{Computing environment:} We implemented all algorithms in Julia using MOSEK's exponential cone solver as well as JuMP to interact with the solver. We used the high performance computing cluster of Imperial College London, which runs a Linux operating system as well as Portable Batch System (PBS) for scheduling the jobs. We ran our experiments as batch jobs on Intel Xeon 2.66GHz processors with 8GB memory in single-core and single-thread mode. The job descriptions as well as all PBS commands are included in the GitHub repository.

\blue{\textbf{Runtimes corresponding to Table~\ref{tab:categorical_results}:} Table~\ref{tab_runtime:categorical_results} presents the mean runtimes of each method on every dataset for the experiments corresponding to Table~\ref{tab:categorical_results}. Here, the times reported for `LR' and `MT' are the times the solver took to solve the corresponding optimization problems. The times corresponding to `DRO' methods we propose are the total solver times summed for each sub-problem solved during the column-and-constraint generation scheme, including the identification of the most violated constraints. The times corresponding to `PI' display the solution time to solve the regularized logistic regression problems plus the time it takes to identify the regularization parameter as proposed in \cite{blanchet_kang_murthy_2019}. The columns of Table~\ref{tab_runtime:categorical_results} are identical to those in Table~\ref{tab:categorical_results}.}
\begin{table}[tb]
{\color{black}
    \centering
    \resizebox{1.2\columnwidth}{!}{%
    \hskip-4.0cm
     \begin{tabular}{lccc|ccc|ccc|ccc}
\toprule
Data Set &  $N$ & $k$ & $m$ & LR  &  DRO ($\kappa=1$) &  DRO ($\kappa=m$)& r-LR  & r-DRO ($\kappa=1$) & r-DRO ($\kappa=m$) & MT ($\kappa = 1$) & MT ($\kappa = m$) & PI ($\alpha = 0.05$)\\
\midrule
     \underline{breast-cancer}  &  277  &  42 & 9 & 0.06 & 218.61 & 214.83 & 0.07 & 15.19 & 20.52 & 0.08 & 0.09 & 0.02 + 0.07 \\
     \underline{spect}  &267& 22 & 22 & 0.08 & 32.05 & 38.95 & 0.07 & 20.81 & 35.50 & 0.07 & 0.09 & 0.02 + 0.03 \\
     \underline{monks-3}  &554 & 11 & 6 & 0.12 & 46.83 & 42.59& 0.08 & 31.04 & 42.53 & 0.12 & 0.14 & 0.03 + 0.02 \\
      \underline{tic-tac-toe}  & 958 & 18 &  9 & 0.20 & 71.26 & 69.80 & 0.13 & 61.88 & 64.33 & 0.23 & 0.27 & 0.06 + 0.04 \\
      \underline{kr-vs-kp}  & 3,196 & 37 & 36 & 1.65 & 217.98 & 40.64 & 0.71 & 21.79 & 11.29 & 2.11 & 2.87 & 0.34 + 0.08\\
      \underline{balance-scale}$\star$ & 625 & 16 & 4 & 0.15 & 13.60 & 17.41 & 0.19 & 19.74 & 24.04 & 0.21 & 0.19 & 0.03 + 0.03 \\
      \underline{hayes-roth}$\star$  & 160 & 11 & 4 & 0.06 & 3.67 & 3.92 & 0.02 & 1.12 & 0.85 & 0.06 & 0.06 & 0.01 + 0.02\\
      \underline{lymphography}$\star$  & 148 & 42 & 18 & 0.09 & 145.46 & 154.69 & 0.03 & 4.83 & 11.91 & 0.08 & 0.09 & 0.01 + 0.07 \\
      \underline{car}$\star$  & 1,728 & 15 &  6 & 0.46 & 168.09 & 112.33 & 0.28 & 210.35 & 181.63 & 0.57 & 0.52 & 0.11 + 0.05 \\
      \underline{splices}$\star$ & 3,189 & 229 & 60 & 4.25 & 4,246.36 & 4,393.97  & 1.48 & 1,169.80 & 453.73 & 6.13 & 10.21 & 1.55 + 1.77 \\
      \underline{house-votes-84}  & 435  & 32 & 16 & 0.18 & 88.56 & 105.66 & 0.06 & 21.6 & 20.71 & 0.12 & 0.16 & 0.03 + 0.08\\
      \underline{hiv}  & 6,590 & 152 & 8 & 8.44 & 2,874.62 & 1,598.83 & 2.20 & 2,119.26 & 768.01 & 10.94 & 12.93 & 2.18 + 1.14 \\
      \underline{primacy-tumor}$\star$  & 339 & 25 & 17 & 0.22 &61.75&57.81& 0.06 & 11.93 & 13.21 & 0.08 & 0.12 & 0.02 + 0.04\\
      \underline{audiology}$\star$ & 226 & 92 & 69 & 0.05 & 18.52 & 16.10 & 0.04 & 9.19 & 5.28 & 0.11 & 0.13 & 0.04 + 0.16 \\
\bottomrule
\end{tabular}} \medskip
    \caption{\color{black} Mean runtimes (in seconds) associated with Table~\ref{tab:categorical_results}.}
    \label{tab_runtime:categorical_results} \vspace{-0.5cm}
}
\end{table}

\textbf{Error bars of the experiments:} Figures~\ref{fig:errors} and~\ref{fig:mixed_errors} report the error distributions corresponding to the experiments of Tables~\ref{tab:categorical_results} and~\ref{tab:mixed_results}, respectively. In these figures, each of the six sub-plots reports the errors of a specific model (\emph{e.g.}, regularized DRO with $\kappa = 1$). In each sub-plot, the horizontal axis lists the considered data sets, and the vertical axis visualizes the {\color{black} 100} test errors in box-and-whisker representation (where the boxes enclose the $25\%$ and $75\%$ quartiles).

\begin{figure}[!htpb]
    \centering
    \includegraphics[width = 0.48\textwidth]{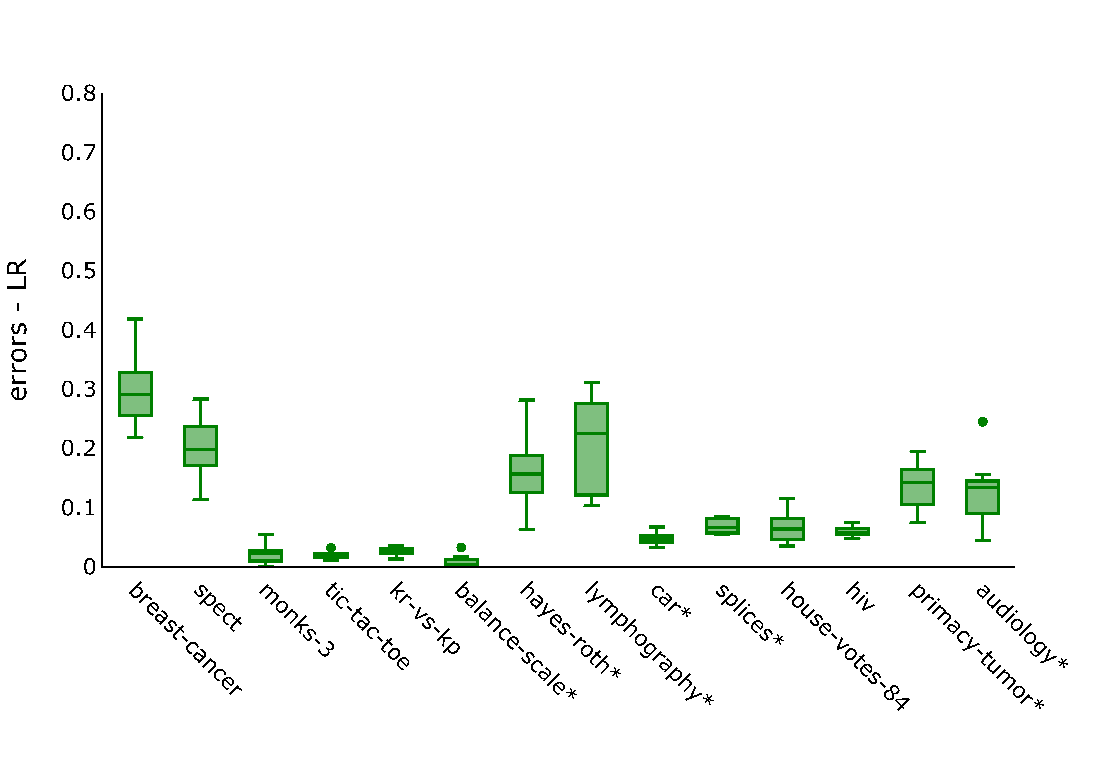} 
    \includegraphics[width = 0.48\textwidth]{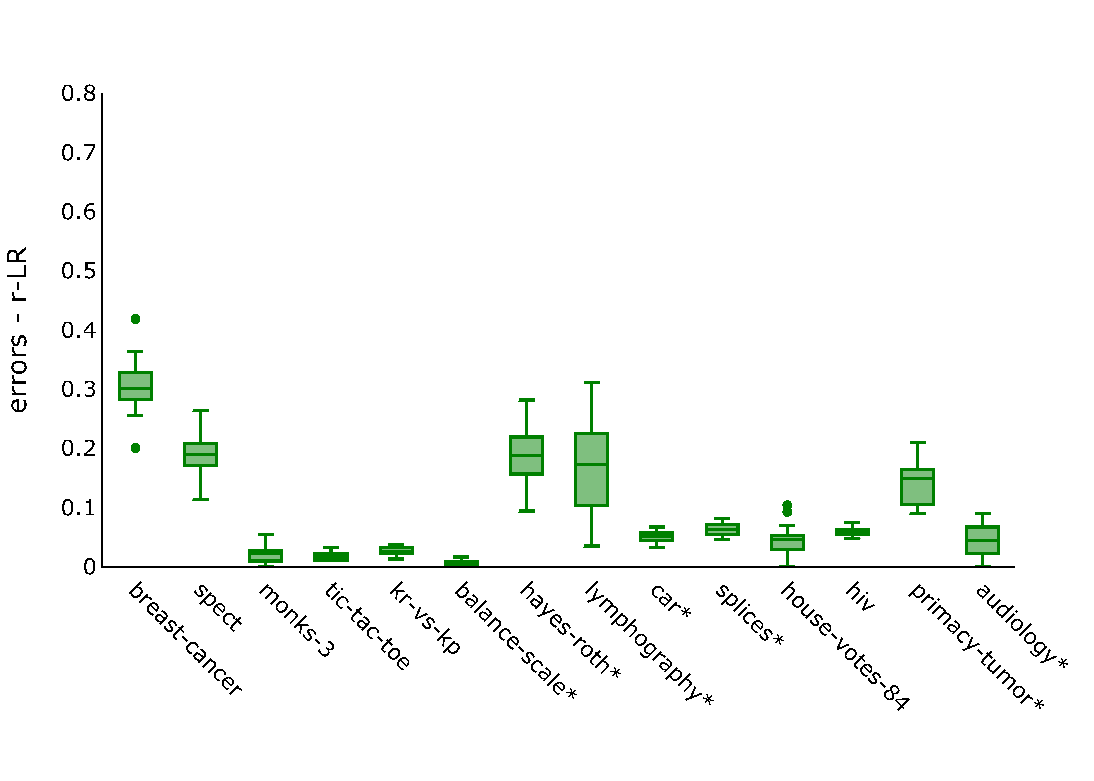}
    \\    \includegraphics[width = 0.48\textwidth]{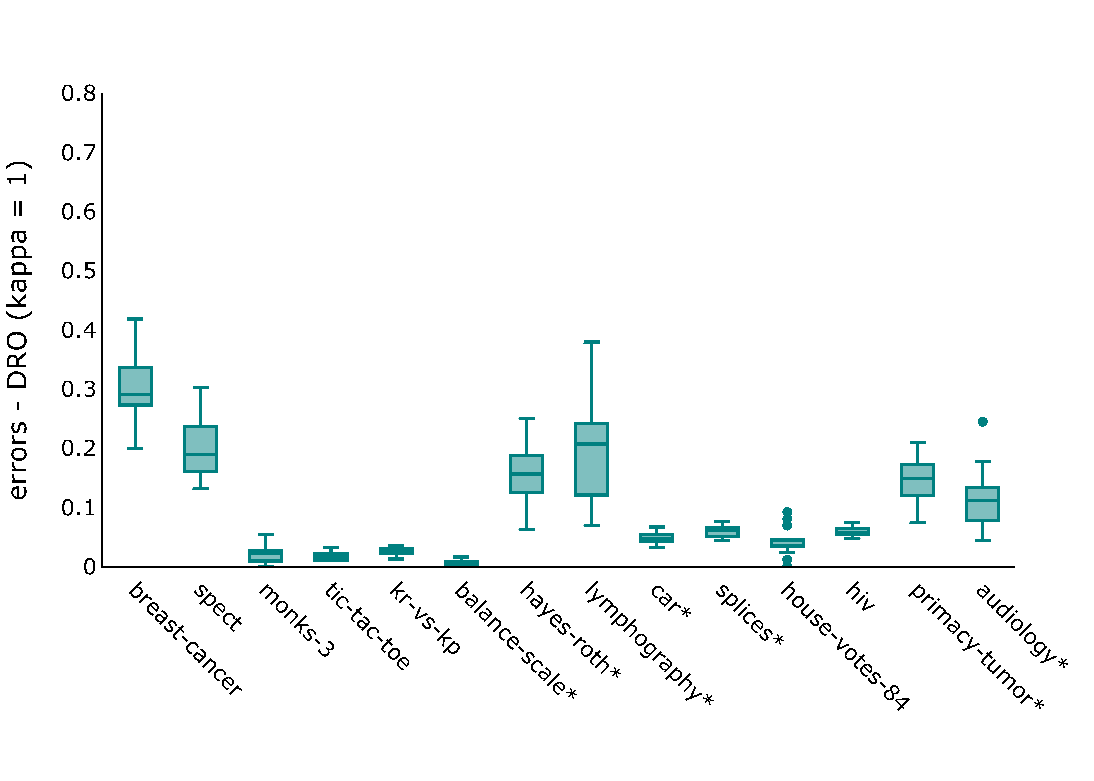}
    \includegraphics[width = 0.48\textwidth]{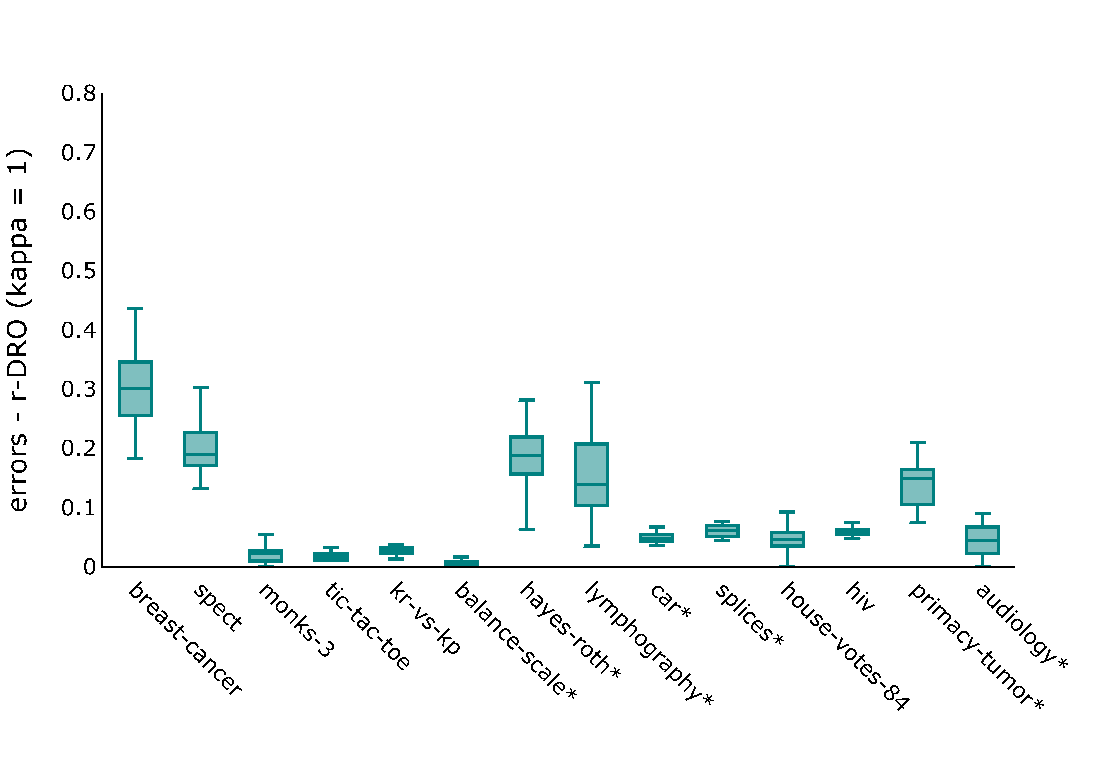}
    \\ 
    \includegraphics[width = 0.48\textwidth]{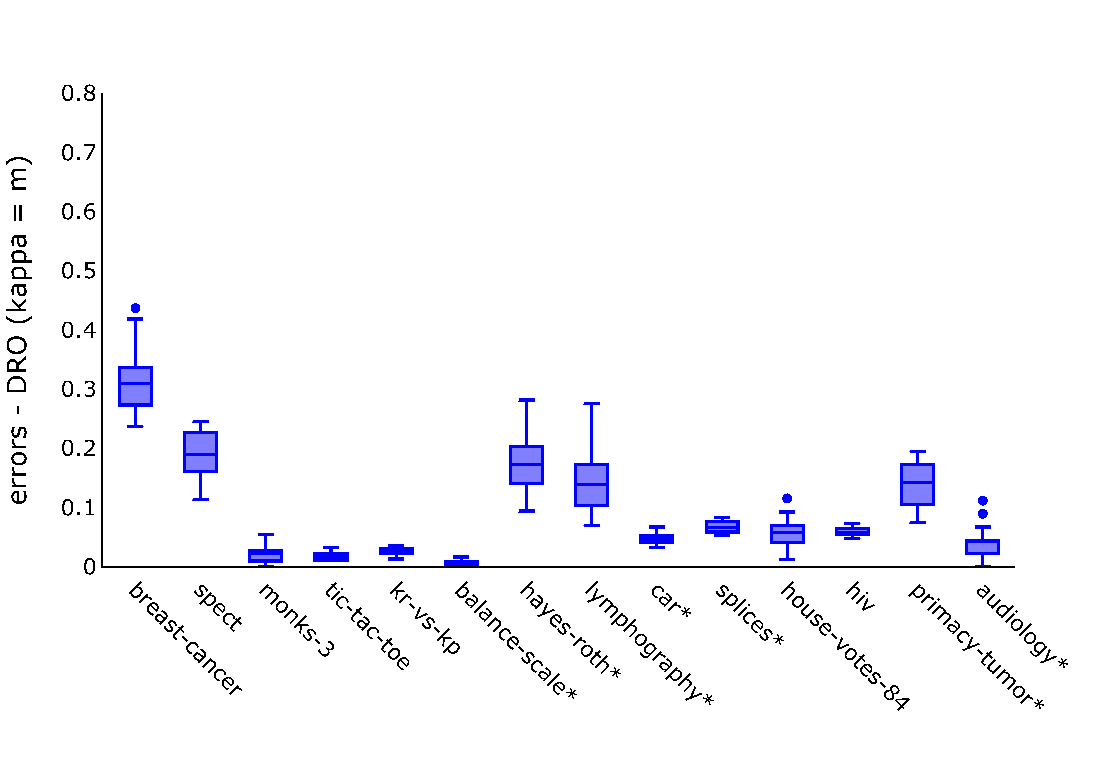}
      \includegraphics[width = 0.48\textwidth]{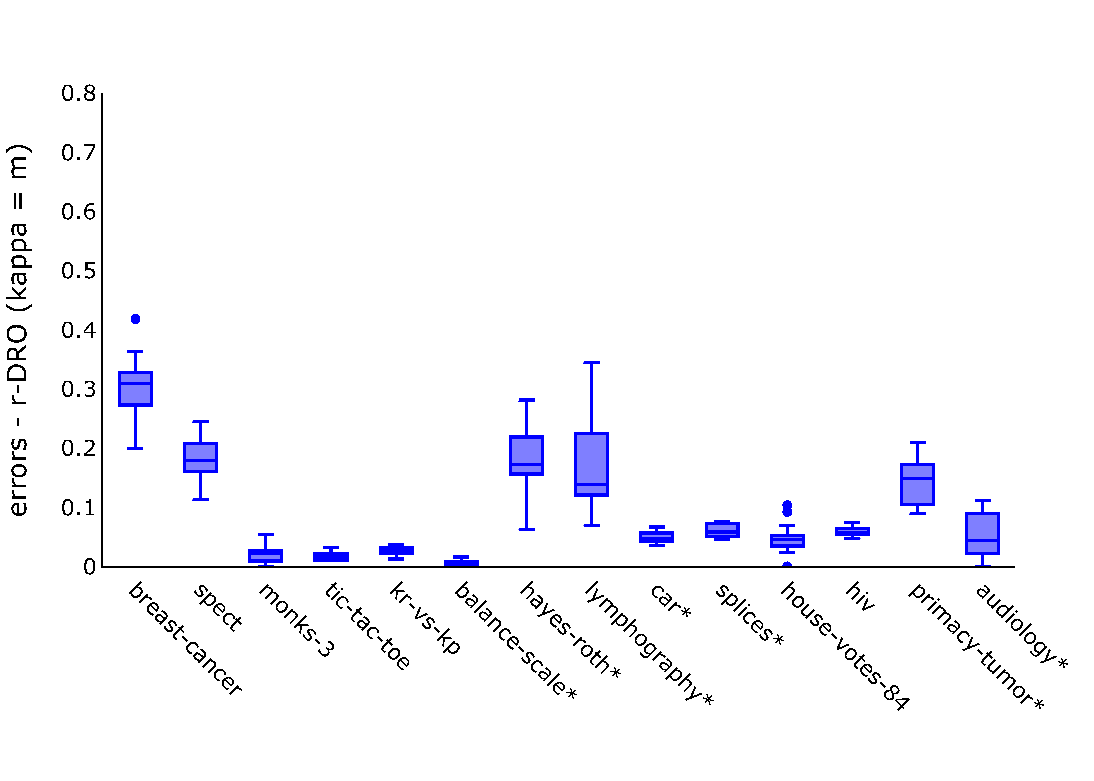}
    \caption{Error bars for the unregularized (left column) and regularized (right column) methods on the considered UCI datasets with categorical features.}
    \label{fig:errors}
\end{figure}

\begin{figure}[!htpb]
    \centering
    \includegraphics[width = 0.48\textwidth]{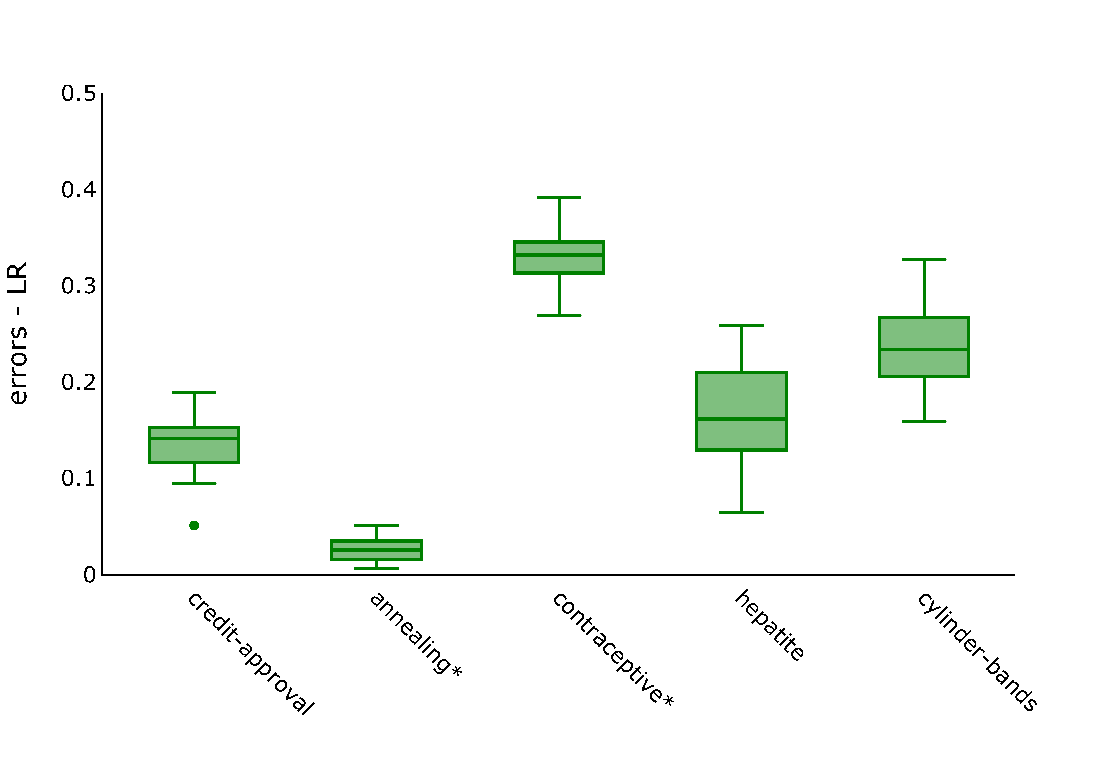} 
    \includegraphics[width = 0.48\textwidth]{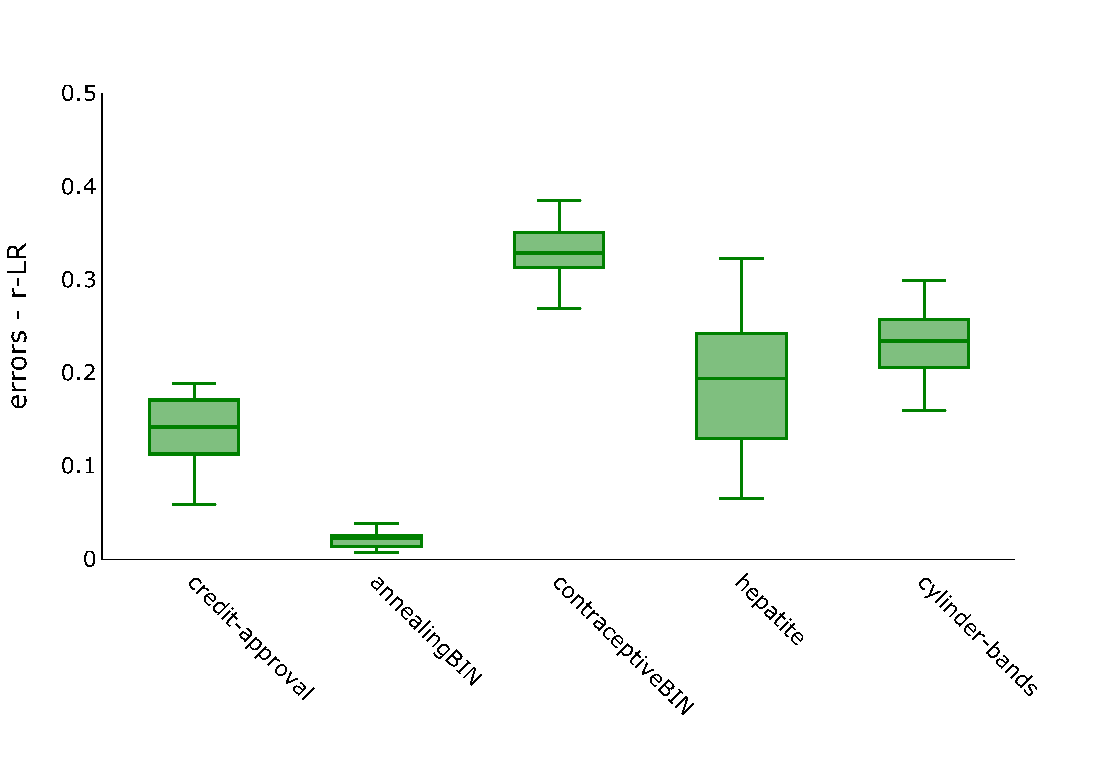}
    \\    \includegraphics[width = 0.48\textwidth]{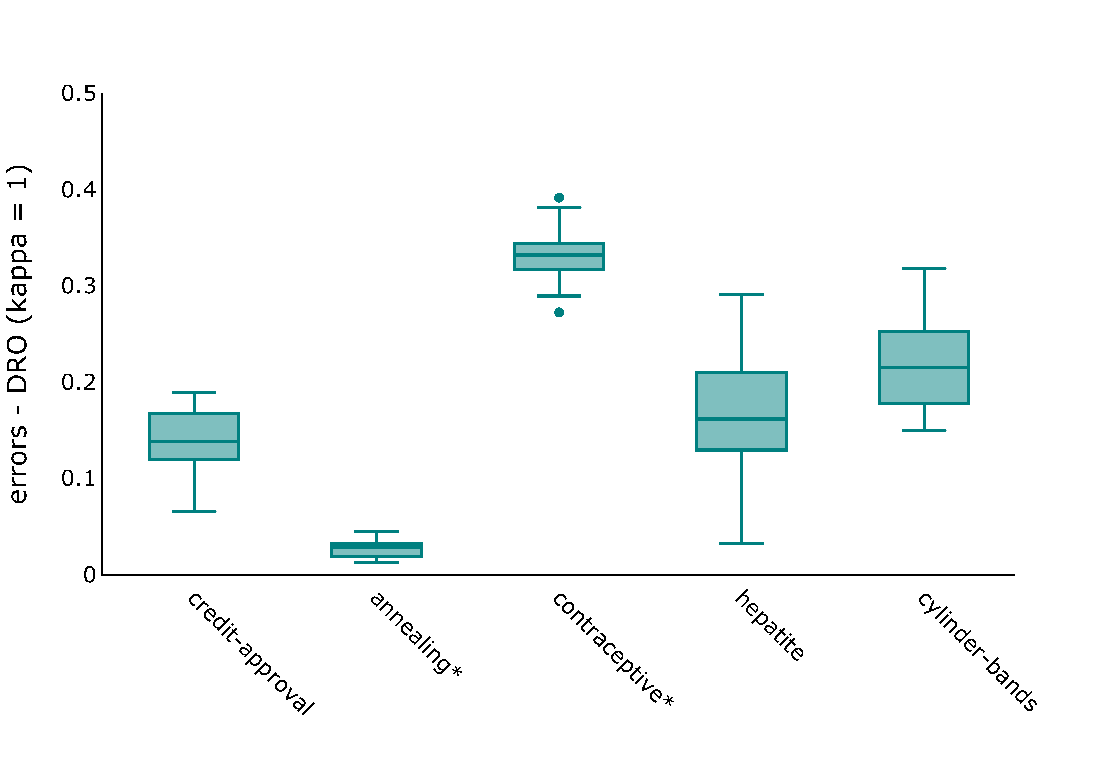}
    \includegraphics[width = 0.48\textwidth]{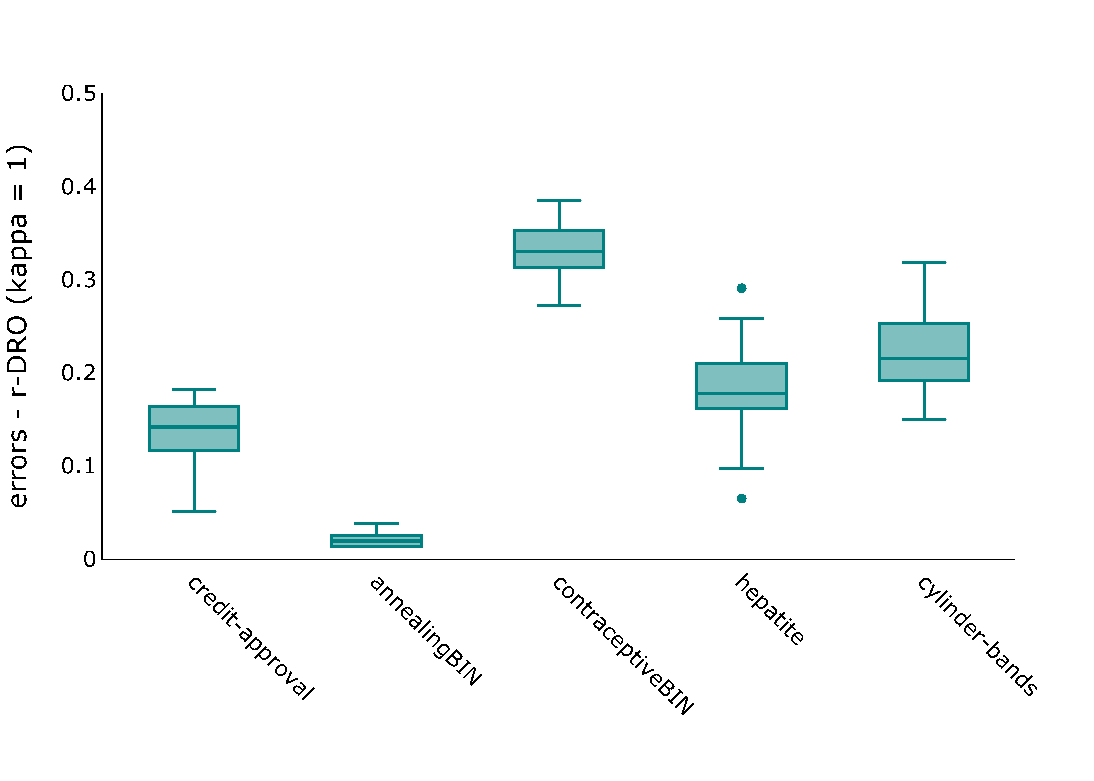}
    \\ 
    \includegraphics[width = 0.48\textwidth]{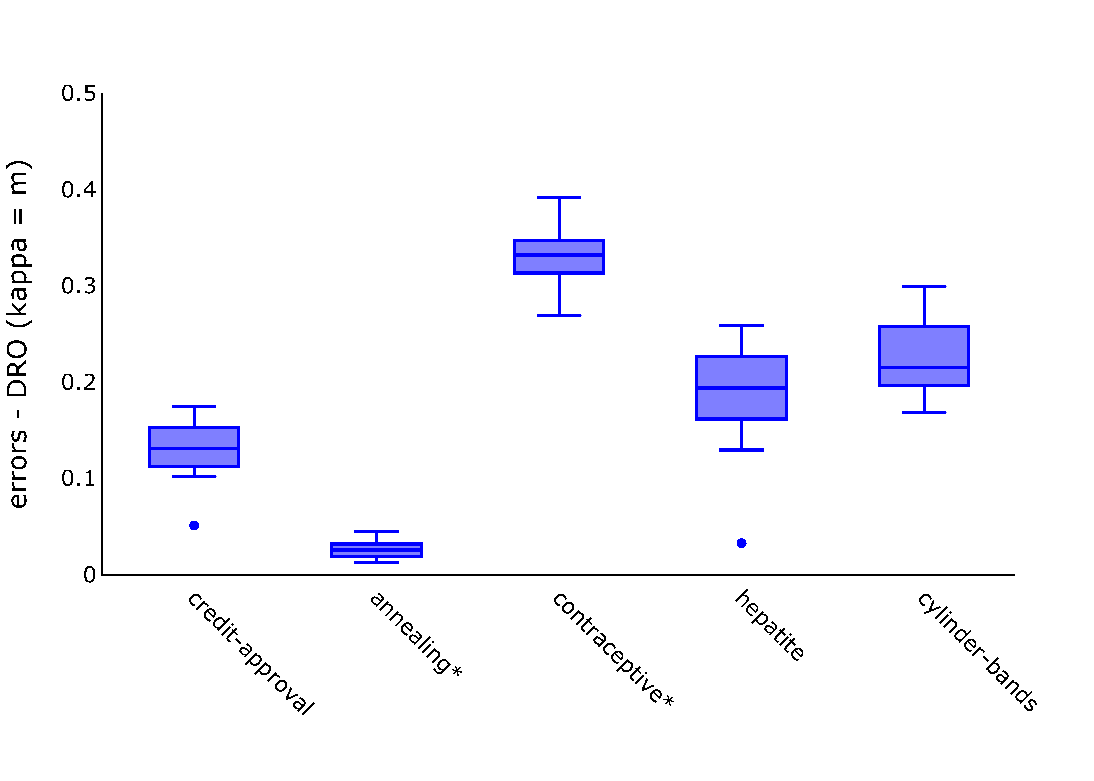}
      \includegraphics[width = 0.48\textwidth]{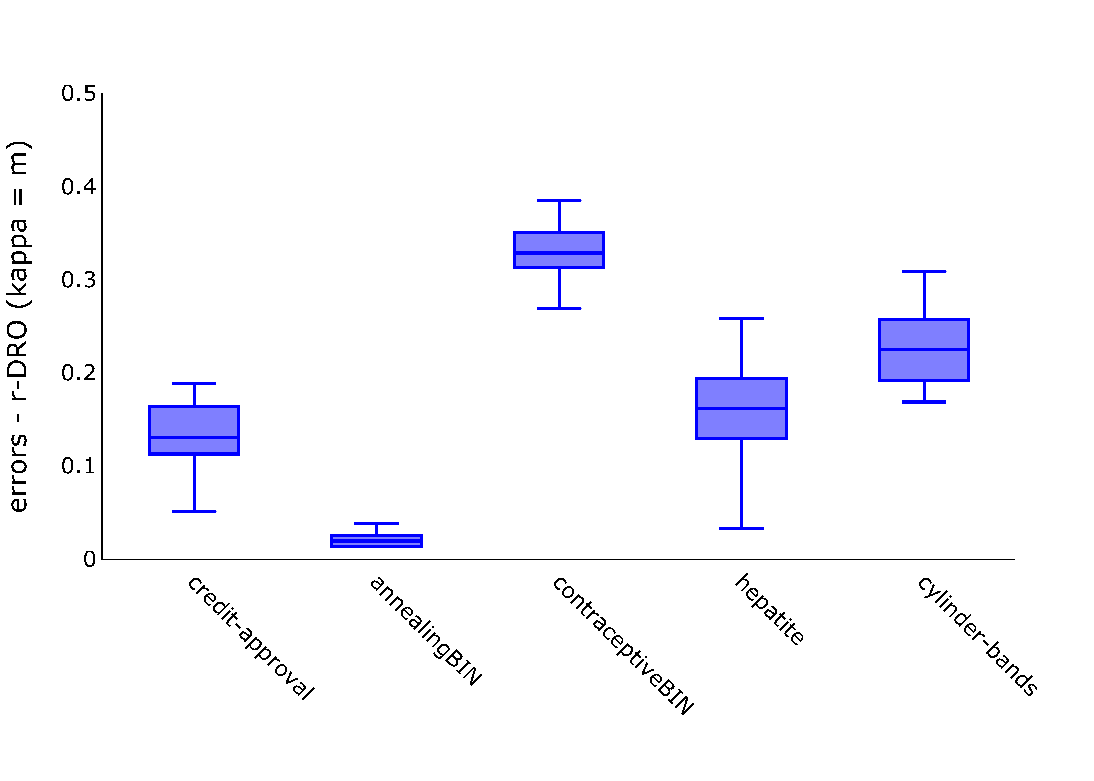}
    \caption{Error bars for the unregularized (left column) and regularized (right column) methods on the considered UCI datasets with mixed features.}
    \label{fig:mixed_errors}
\end{figure}

\clearpage
\newpage

\section*{Appendix E: Proofs}

Our proof of Theorem~\ref{thm:coneheads} relies on the following result from the literature, which we state first.

\begin{lem}\label{lem:ConvexLemma}
    Consider the convex function $h_{\bgamma}(\bw):= \log \left(1+ \exp(-\bgamma^\top \bw +v) \right)$ for $\bgamma, \, \bw \in \mR^n$ and $v \in \mR$. Then
    \[
    \sup_{\bw \in \mR^n} \, h_{\bgamma}(\bw) - \lambda ||\widehat{\bw} - \bw || = \begin{cases}
			h_{\bgamma}(\widehat{\bw}) & \text{\emph{if} } \ ||\bgamma||_{*} \leq \lambda, \\
             + \infty & \mbox{\emph{otherwise}}
		 \end{cases}
    \]
    for every $\lambda >0$, where $||\cdot||_{*}$ is the dual norm of $|| \cdot ||$, that is, $||\bgamma||_{*} := \sup_{||\bw || \leq 1} \bgamma^\top \bw$.
\end{lem}

\begin{proof}
    The statement immediately follows from Lemma~47 of \cite{JMLR:v20:17-633}.
\end{proof}

Our proof of Theorem~\ref{thm:coneheads} follows \cite{NIPS2015} with two main changes: \emph{(i)} we handle the categorical features in $\bz$, and \emph{(ii)} we reformulate the optimization problem as an exponential conic program. Despite the similarities to the existing result, we include the entire proof to keep the paper self-contained.

{\bf Proof of Theorem \ref{thm:coneheads}:} Recall that $\bxi = (\bm{x}, \bm{z}, y) \in \Xi = \mathbb{R}^n \times \mathbb{C} \times \{ -1, +1 \}$. To simplify our notation, we will use $\bxi$ and $(\bx,\bz,y)$ interchangeably. The inner problem in~\eqref{eq:the_mother_of_all_problems} can be written as
\begin{equation*}
    \left[
    \begin{array}{l@{\quad}l}
        \displaystyle \mathop{\text{maximize}}_{\mathbb{Q}} & \displaystyle \mathbb{E}_\mathbb{Q} \left[ l_{\bm{\beta}} (\bm{x}, \bm{z}, y) \right] \\
        \displaystyle \text{subject to} & \displaystyle \mathbb{Q} \in \mathfrak{B}_\epsilon (\widehat{\mathbb{P}}_N)
    \end{array}
    \right]
    \quad = \quad
    \left[
    \begin{array}{l@{\quad}l}
        \displaystyle \mathop{\text{maximize}}_{\mathbb{Q}} & \displaystyle \int_{\bxi \in \Xi} l_{\bbeta}(\bxi) \, \mQ( \diff \bxi ) \\
        \displaystyle \text{subject to} & \displaystyle \mathbb{Q} \in \mathfrak{B}_\epsilon (\widehat{\mathbb{P}}_N)
    \end{array}
    \right],
\end{equation*}
and replacing $\mathfrak{B}_\epsilon (\widehat{\mathbb{P}}_N)$ with its definition (\emph{cf.}~Definition~\ref{def:wasserstein}) yields
\begin{equation*}
    \begin{array}{l@{\quad}l@{\qquad}l}
        \displaystyle \mathop{\text{maximize}}_{\mathbb{Q}, \Pi} & \displaystyle \int_{\bxi \in \Xi} l_{\bbeta}(\bxi) \, \mQ( \diff \bxi ) \\[5mm]
        \displaystyle \text{subject to} & \displaystyle \int_{(\bxi, \bxi') \in \Xi^2} d(\bxi, \bxi') \, \Pi(\diff \bxi, \diff \bxi') \leq \epsilon \\[5mm]
        & \displaystyle \int_{\bxi \in \Xi} \Pi(\diff \bxi, \diff \bxi') =  \widehat{\mathbb{P}}_N(\diff \bxi') & \forall \bxi' \in \Xi \\[5mm]
        & \displaystyle \int_{\bxi' \in \Xi} \Pi(\diff \bxi, \diff \bxi') =  {\mathbb{Q}}(\diff \bxi) & \forall \bxi \in \Xi \\[5mm]
        & \displaystyle \mathbb{Q} \in \mathcal{P}_0 (\Xi), \;\; \Pi \in \mathcal{P}_0 (\Xi^2).
    \end{array}
\end{equation*}
We can substitute $\mathbb{Q}$ using the second equality constraint to obtain
\begin{equation*}
    \begin{array}{l@{\quad}l@{\qquad}l}
        \displaystyle \mathop{\text{maximize}}_{\Pi} & \displaystyle \int_{\bxi \in \Xi} l_{\bbeta}(\bxi) \int_{\bxi' \in \Xi} \Pi(\diff \bxi, \diff \bxi') \\[5mm]
        \displaystyle \text{subject to} & \displaystyle \int_{(\bxi, \bxi') \in \Xi^2} d(\bxi, \bxi') \, \Pi(\diff \bxi, \diff \bxi') \leq \epsilon \\[5mm]
        & \displaystyle \int_{\bxi \in \Xi} \Pi(\diff \bxi, \diff \bxi') =  \widehat{\mathbb{P}}_N(\diff \bxi') & \forall \bxi' \in \Xi \\[5mm]
        & \displaystyle \Pi \in \mathcal{P}_0 (\Xi^2).
    \end{array}
\end{equation*}
Denoting by $\mQ^i(\diff \bxi) := \Pi(\diff \bxi | \bxi^i)$ the conditional distribution of $\Pi$ upon the realization of $\bxi' = \bxi^i$ and exploiting the fact that $\widehat{\mathbb{P}}_N$ is a discrete distribution supported on the $N$ atoms $\{\bxi^i\}_{i \in [N]}$, we can use the marginalized representation $ \Pi(\diff \bxi,\diff \bxi') =\frac{1}{N} \sum_{i=1}^N \delta_{\bxi^i}(\diff \bxi')\mQ^i(\diff \bxi)$ to obtain the equivalent reformulation
\begin{equation*}
    \begin{array}{l@{\quad}l@{\qquad}l}
        \displaystyle \mathop{\text{maximize}}_{\mQ^i} & \displaystyle \dfrac{1}{N} \sum_{i=1}^N \int_{\bxi \in \Xi} l_{\bbeta}(\bxi) \, \mQ^i (\diff \bxi) \\[5mm]
        \displaystyle \text{subject to} & \displaystyle \dfrac{1}{N} \sum_{i=1}^N \int_{\bxi \in \Xi} d(\bxi, \bxi^i) \, \mQ^i (\diff \bxi) \leq \epsilon \\[5mm]
        & \displaystyle \mQ^i \in \mathcal{P}_0(\Xi), \, i \in [N].
    \end{array}
\end{equation*}
We can now decompose the $N$ variables (distributions) $\mQ^i$ into $2N \cdot \prod_j k_j$ measures $\mQ^i_{\bz, y}$:
\begin{equation*}
    \begin{array}{l@{\quad}l@{\qquad}l}
        \displaystyle \mathop{\text{maximize}}_{\mQ^i_{\bz, y}} & \displaystyle \dfrac{1}{N} \sum_{i=1}^N \sum_{(\bz, y) \in \mathbb{C} \times \{ -1, +1\}} \int_{\bx \in \mathbb{R}^n} l_{\bbeta}(\bxi) \, \mQ^i_{\bz, y} (\diff \bx) \\[5mm]
        \displaystyle \text{subject to} & \displaystyle \dfrac{1}{N} \sum_{i=1}^N \sum_{(\bz, y) \in \mathbb{C} \times \{ -1, +1\}} \int_{\bx \in \mathbb{R}^n} d(\bxi, \bxi^i) \, \mQ^i_{\bz, y}(\diff \bx) \leq \epsilon \\[5mm]
        & \displaystyle \sum_{(\bz, y) \in \mathbb{C} \times \{ -1, +1\}} \int_{\bx \in \mathbb{R}^n}  \mQ^i_{\bz, y}(\diff \bx) = 1 & \forall i \in [N] \\ 
        & \displaystyle \mQ^i_{\bz, y} \in \mathcal{M}_+ (\mathbb{R}^n), \, i \in [N] \text{ and } (\bz, y) \in \mathbb{C} \times \{ -1, +1\},
    \end{array}
\end{equation*}
where $\mathcal{M}_+ (\mathbb{R}^n)$ denotes the space of non-negative measures supported on $\mathbb{R}^n$. This infinite-dimensional linear program admits the dual
\begin{equation*}
    \begin{array}{l@{\quad}l@{\qquad}l}
        \displaystyle \mathop{\text{minimize}}_{\lambda, \bm{u}} & \displaystyle \lambda \epsilon + \sum_{i=1}^N u_i \\[5mm]
        \displaystyle \text{subject to} & \displaystyle \dfrac{1}{N} \cdot \underset{\bx \in \mathbb{R}^n}{\sup} \left\{ l_{\bbeta}(\bx, \bz, y) - \lambda d((\bx, \bz, y), \bxi^i) \right\} \leq u_i & \displaystyle \forall i \in [N] \\
        & & \displaystyle \forall (\bz, y) \in \mathbb{C} \times \{-1, +1\} \\
        & \displaystyle \lambda \geq 0, \;\; \bm{u} \in \mathbb{R}^N.
    \end{array}
\end{equation*}
Strong duality holds for any $\epsilon > 0$ due to Proposition~3.4 of \cite{Shapiro:2001}. We rewrite the dual problem further by substituting $s_i := Nu_i$, using Definition~\ref{def:metric} to write out $d$ explicitly, as well as breaking down the dual constraints into separate cases for $y = -1$ and $y = +1$:
\begin{equation*}
    \mspace{-35mu}
    \begin{array}{l@{\quad}l@{\qquad}l}
        \displaystyle \mathop{\text{minimize}}_{\lambda, \bm{s}} & \displaystyle \lambda \epsilon + \dfrac{1}{N}\sum_{i=1}^N s_i \\[5mm]
        \displaystyle \text{subject to} & \displaystyle \underset{\bx \in \mathbb{R}^n}{\sup} \{ l_{\bbeta}(\bx, \bz, +1) - \lambda \lVert \bx-\bx^i \rVert\} - \lambda \kappa \cdot \indicate{y^i \neq 1} - \lambda d_{\text{C}} (\bz,\bz^i)  \leq s_i &  \forall i \in [N],\; \bz \in \mathbb{C} \\
        & \displaystyle \underset{\bx \in \mathbb{R}^n}{\sup} \{ l_{\bbeta}(\bx, \bz, -1) - \lambda \lVert \bx-\bx^i \rVert\} - \lambda \kappa \cdot \indicate{y^i \neq -1} - \lambda d_{\text{C}} (\bz,\bz^i)  \leq s_i &  \forall i \in [N],\; \bz \in \mathbb{C} \\
        & \displaystyle \lambda \geq 0, \;\; \bm{s} \in \mathbb{R}^N
    \end{array}
\end{equation*}
Applying Lemma~\ref{lem:ConvexLemma} to the suprema of the above problem results in the reformulation
\begin{equation*}
    \begin{array}{l@{\quad}l@{\qquad}l}
        \displaystyle \mathop{\text{minimize}}_{\lambda, \bm{s}} & \displaystyle \lambda \epsilon + \dfrac{1}{N} \sum_{i=1}^N s_i \\[5mm]
        \displaystyle \text{subject to} & \displaystyle l_{\bbeta}(\bx^i, \bz, +1) - \lambda \kappa \cdot \indicate{y^i \neq 1} - \lambda d_{\text{C}} (\bz,\bz^i)  \leq s_i &  \forall i \in [N], \; \forall \bz \in \mathbb{C} \\
        & \displaystyle l_{\bbeta}(\bx^i, \bz, -1) - \lambda \kappa \cdot \indicate{y^i \neq -1} - \lambda d_{\text{C}} (\bz,\bz^i)  \leq s_i &  \forall i \in [N], \; \forall \bz \in \mathbb{C} \\
        & \displaystyle \left \lVert \bm{\beta}_{\text{N}} \right \rVert_* \leq \lambda \\
        & \displaystyle \lambda \geq 0, \;\; \bm{s} \in \mathbb{R}^N.
    \end{array}
\end{equation*}
For $y^i = +1$, the identities
\begin{align*}
    l_{\bbeta}(\bx^i, \bz, +1) - \lambda \kappa \cdot \indicate{y^i \neq 1} \;\; &= \;\; l_{\bbeta}(\bx^i, \bz, y^i) \\
    l_{\bbeta}(\bx^i, \bz, -1) - \lambda \kappa \cdot \indicate{y^i \neq -1} \;\; &= \;\; l_{\bbeta}(\bx^i, \bz, -y^i) - \lambda \kappa
\end{align*}
hold; similar identities hold for $y^i = -1$. The above optimization problem thus simplifies to
\begin{equation}\label{eq:reform_for_absence_of_regularizer_proof}
    \begin{array}{l@{\quad}l@{\qquad}l}
        \displaystyle \mathop{\text{minimize}}_{\lambda, \bm{s}} & \displaystyle \lambda \epsilon + \dfrac{1}{N} \sum_{i=1}^N s_i \\[5mm]
        \displaystyle \text{subject to} & \displaystyle l_{\bbeta}(\bx^i, \bz, y^i) - \lambda d_{\text{C}} (\bz,\bz^i)  \leq s_i &  \forall i \in [N], \; \forall \bz \in \mathbb{C} \\
        & \displaystyle l_{\bbeta}(\bx^i, \bz, -y^i) - \lambda\kappa - \lambda d_{\text{C}} (\bz,\bz^i)  \leq s_i &  \forall i \in [N], \; \forall \bz \in \mathbb{C} \\
        & \displaystyle \left \lVert \bm{\beta}_{\text{N}} \right \rVert_* \leq \lambda \\
        & \displaystyle \lambda \geq 0, \;\; \bm{s} \in \mathbb{R}^N.
    \end{array}
\end{equation}
Plugging this problem into the overall optimization problem over $\bm{\beta} \in \mathbb{R}^{1 + n + k}$ therefore yields
\begin{equation}\label{eq:conehead_equiv_formulation}
    \begin{array}{l@{\quad}l@{\qquad}l}
        \displaystyle \mathop{\text{minimize}}_{\bm{\beta}, \lambda, \bm{s}} & \displaystyle \lambda \epsilon + \dfrac{1}{N} \sum_{i=1}^N s_i \\[5mm]
        \displaystyle \text{subject to} & \displaystyle l_{\bbeta}(\bx^i, \bz, y^i) - \lambda d_{\text{C}} (\bz,\bz^i)  \leq s_i &  \forall i \in [N], \; \forall \bz \in \mathbb{C} \\
        & \displaystyle l_{\bbeta}(\bx^i, \bz, -y^i) - \lambda\kappa - \lambda d_{\text{C}} (\bz,\bz^i)  \leq s_i &  \forall i \in [N], \; \forall \bz \in \mathbb{C} \\
        & \displaystyle \left \lVert \bm{\beta}_{\text{N}} \right \rVert_* \leq \lambda \\
        & \multicolumn{2}{l}{\displaystyle \mspace{-8mu} \bm{\beta} = (\beta_0, \bm{\beta}_{\text{N}}, \bm{\beta}_{\text{C}}) \in \mathbb{R}^{1 + n + k}, \;\; \lambda \geq 0, \;\; \bm{s} \in \mathbb{R}^N.}
    \end{array}
\end{equation}
Finally, we discuss how to reformulate the \emph{softplus constraints} (\emph{i.e.}, convex constraints with a log-loss function on the left-hand side and a linear function on the right-hand side) as exponential cone constraints. To this end, recall that the \textit{exponential cone} is defined as
\begin{align*}
    \mathcal{K}_{\exp} := \mathrm{cl} \left(\left\{(a,b,c) \ : \ a \geq b \exp(c / b), \ a > 0, \ b > 0 \right\}\right) \subset \mathbb{R}^3,
\end{align*}
where $\mathrm{cl}(\cdot)$ denotes the closure. Observe further that
\begin{align*}
    \log(1 + \exp(c)) \leq a \iff 1 + \exp(c) \leq \exp(a) \iff \exp(-a) + \exp(c - a) \leq 1.
\end{align*}
Using the auxiliary variables $u, v$, we can reformulate this constraint as
\begin{align*}
    \left[
    \begin{array}{l}
        u + v \leq 1 \\
        \exp(-a) \leq u \\
        \exp(c - a) \leq v
    \end{array}
    \right]
    \quad \iff \quad
    \left[
    \begin{array}{l}
        u + v \leq 1 \\
        (u, 1, -a) \in \mathcal{K}_{\exp} \\
        (v, 1, c-a) \in \mathcal{K}_{\exp}
    \end{array}
    \right],
\end{align*}
where the second system of equations uses the definition of $\mathcal{K}_{\exp}$. Applying this reformulation to both softplus constraints in our minimization problem results in problem~\eqref{eq:cone_head_reformulation} and thus concludes the proof of the theorem.
$\Box$\\

\begin{rem}[Exponential Cone Reformulation]
    The use of exponential conic constraints has a significant impact on the theoretical complexity of problem~\eqref{eq:cone_head_reformulation}. We refer to \cite{serrano2015algorithms} for an overview of exponential conic programs and to \cite{mosek2010mosek} for modeling techniques, respectively. We also note that the difference in theoretical complexity carries over to a significant difference in practical solvability: In our experiments, we observe that using the exponential cone solver of MOSEK drastically speeds up the solution times of our DR logistic regression problem.
\end{rem}

{\bf Proof of Theorem \ref{thm:complexity}:}
In view of the first statement, we recall that the strongly NP-hard integer programming problem is defined as follows \cite{GJ79:ComputersIntractability}:
\begin{center}
    \fbox{\parbox{10.25cm}{
    {\centering \textsc{0/1 Integer Programming.} \\}
    \textbf{Instance.} Given are $\bm{F} \in \mathbb{Z}^{\mu \times \nu}$, $\bm{g} \in \mathbb{Z}^\mu$, $\bm{c} \in \mathbb{Z}^\nu$, $\zeta \in \mathbb{Z}$. \\
    \textbf{Question.} Is there a vector $\bm{\chi} \in \left\{ 0, 1 \right\}^\nu$ such that $\bm{F} \bm{\chi} \leq \bm{g}$ and $\bm{c}^\top \bm{x} \geq \zeta$?}
    }
\end{center}
{\color{black}Here, $\mathbb{Z}$ denotes the set of integers.} We claim that the answer to the integer programming problem is affirmative if and only if the optimal objective value for the DR logistic regression~\eqref{eq:the_mother_of_all_problems} with $n = 0$, $m = \nu$, $k_1 = \ldots = k_m = 2$, $N = 1$ and the loss function
\begin{equation*}
    l_{\bm{\beta}} (\bm{z}, y)
    \; = \;
    \begin{cases}
        \bm{c}^\top \bm{z} & \text{if } \bm{F} \bm{z} \leq \bm{g}, \\
        -\mathrm{M} & \text{otherwise}
    \end{cases}
\end{equation*}
and $\mathrm{M}$ and $\epsilon$ sufficiently large is greater than or equal to $\zeta$. This then immdiately implies that the DR logistic regression~\eqref{eq:the_mother_of_all_problems} is strongly NP-hard for generic loss functions (which include the above loss function as a special case).

To see this, note that our reformulation from the proof of Theorem~\ref{thm:coneheads}, under the assumption that $n = 0$ and $N = 1$, shows that problem~\eqref{eq:the_mother_of_all_problems} is equivalent to
\begin{equation*}
    \begin{array}{l@{\quad}l@{\qquad}l}
        \displaystyle \mathop{\text{minimize}}_{\bm{\beta}, \lambda, s} & \displaystyle \lambda \epsilon + s \\[5mm]
        \displaystyle \text{subject to} & \displaystyle l_{\bbeta}(\bz, y^1) - \lambda d_{\text{C}} (\bz,\bz^1)  \leq s &  \forall \bz \in \mathbb{C} \\
        & \displaystyle l_{\bbeta}(\bz, -y^1) - \lambda\kappa - \lambda d_{\text{C}} (\bz,\bz^1)  \leq s & \forall \bz \in \mathbb{C} \\
        & \multicolumn{2}{l}{\displaystyle \mspace{-8mu} \bm{\beta} = (\beta_0, \bm{\beta}_{\text{C}}) \in \mathbb{R}^{k + 1}, \;\; \lambda \geq 0, \;\; s \in \mathbb{R}.}
    \end{array}
\end{equation*}

Assuming that $\epsilon$ is sufficiently large, the optimal $\lambda$ vanishes, and the problem simplifies further to

\begin{equation*}
    \begin{array}{cll}
        \displaystyle \mathop{\text{minimize}}_{\bm{\beta}, s} & \displaystyle s &  \\[5mm]
        \displaystyle \text{subject to} & \bm{c}^\top \bm{z} \leq s & \forall \bz \in \mathbb{C} \, : \, \bm{F} \bm{z} \leq \bm{g} \\
        & -\mathrm{M} \leq s & \text{if } \exists \bz \in \mathbb{C} \, : \, \bm{F} \bm{z} \not\leq \bm{g} \\
        & \multicolumn{2}{l}{\displaystyle \mspace{-8mu} \bm{\beta} = (\beta_0, \bm{\beta}_{\text{C}}) \in \mathbb{R}^{k + 1}, \;\; s \in \mathbb{R}} 
    \end{array}
\end{equation*}
where we have also plugged in the aforementioned loss function. One readily observes that the optimal value of this problem equals $\max \{ \bm{c}^\top \bm{\chi} \, : \, \bm{F} \bm{\chi} \leq \bm{g}, \; \bm{\chi} \in \{ 0, 1 \}^\nu \}$ whenever $\{ \bm{F} \bm{\chi} \leq \bm{g}, \; \bm{\chi} \in \{ 0, 1 \}^\nu \}$ is non-empty and $-\mathrm{M}$ otherwise. In particular, the optimal value of the problem is greater than or equal to $\zeta$ if and only if the answer to the integer programming problem is affirmative.

In view of the second statement, we note that the feasible region of problem~\eqref{eq:the_mother_of_all_problems} can be circumscribed by a convex body as per Definition~2.1.16 of \cite{GLS88:geom_algos}, and that Algorithm~\ref{alg:most_violated} in Section~\ref{sec:cac_generation} constitutes a weak separation oracle for problem~\eqref{eq:the_mother_of_all_problems}. The statement thus follows from Corollary~4.2.7 of \cite{GLS88:geom_algos}.
\qed \\

{\bf Proof of Theorem \ref{thm:it_aint_regularized}:}
Consider a class of instances of problem~\eqref{eq:the_mother_of_all_problems} with a Wasserstein radius $\epsilon < 1$, $n = 1$ numerical feature, and $m = 1$ categorical feature with $k_1 = 2$ so that the categorical feature is binary. We have $N = 1$ sample with $z^1_1 = 0$ and $y^1_1 = -1$. This instance class is therefore parameterized by the value of the numerical feature $x^1_1$ in the data sample. We now derive the optimal objective value of problem~\eqref{eq:the_mother_of_all_problems} analytically as a function of $x^1_1$, and show that this functional dependence on $x^1_1$ cannot be represented by the objective function of \emph{any} regularized logistic regression.

For our instance class, the reformulation of our DR logistic regression~\eqref{eq:the_mother_of_all_problems} simplifies to
\begin{equation}\label{counter-example}
    \begin{array}{l@{\quad}l@{\qquad}l}
        \displaystyle \mathop{\text{minimize}}_{\lambda, s_1,\bbeta} & \displaystyle \lambda \epsilon + s_1 \\
        \displaystyle \text{subject to} & \displaystyle l_{\bbeta}(x^1_1, z, -1) - \lambda z \leq s_1 & \forall z \in \mathbb{B} \\
        & \displaystyle l_{\bbeta} (x^1_1, z, 1) - \lambda \kappa - \lambda z \leq s_1 & \forall z \in \mathbb{B} \\
        & \displaystyle \lambda \geq \lvert \bbeta_{\text{N}} \rvert, \;\; s_1 \in \mathbb{R}.
    \end{array}
\end{equation}
Assuming further that $\kappa$ approaches $\infty$, the constraint $\lambda \geq 0$ (implied by the constraint $\lambda \geq \lvert \bbeta_{\text{N}} \rvert)$ implies that the second constraint is redundant. Removing this second constraint and substituting $s_1$ into the objective function results in
\begin{equation}  \label{eq:counter-example2}
    \begin{array}{l@{\quad}l@{\qquad}l}
        \displaystyle \mathop{\text{minimize}}_{\lambda,\bbeta} & \displaystyle \lambda \epsilon + \max \left\{ l_{\bbeta}(x^1_1, 0, -1), \; l_{\bbeta}(x^1_1, 1, -1) - \lambda \right\} \\
        \displaystyle \text{subject to} &\displaystyle \lambda \geq \lvert \bbeta_{\text{N}} \rvert.
    \end{array}
\end{equation}
We now break the minimization over $(\lambda, \,\bbeta)$ in (\ref{eq:counter-example2}) into a minimization over $\lambda$ first followed by the minimization over $\bbeta$. That leads to re-writing (\ref{eq:counter-example2}) as
\begin{equation} \label{eq:counter-example3}
    \begin{array}{l@{\quad}l@{\qquad}l}
        \displaystyle \mathop{\text{minimize}}_{\bbeta} & f(\bbeta,x^1_1)
    \end{array}
\end{equation}
where
\begin{equation} \label{eq:counter-example4}
    f(\bbeta,x_1^1) = \left\{ \begin{array}{l@{\quad}l@{\qquad}l}
        \displaystyle \mathop{\text{minimize}}_{\lambda} & \lambda \epsilon + \max \left\{ l_{\bbeta}(x^1_1, 0, -1), \; l_{\bbeta}(x^1_1, 1, -1) - \lambda \right\} \\
        \displaystyle \text{subject to} &\displaystyle \lambda \geq \lvert \bbeta_{\text{N}} \rvert
    \end{array} \right..
\end{equation}
Since the Wasserstein radius satisfies $\epsilon < 1$, it is possible to achieve a reduction in the objective function choosing values larger than $\lvert \bbeta_{\text{N}} \rvert$ for $\lambda$. Focusing on the minimization over $\lambda$ in (\ref{eq:counter-example4}), we recognize that the optimal solution is $\lambda^\star = g(\bbeta,x^1_1):= l_{\bbeta}(x^1_1, 1, -1) - l_{\bbeta}(x^1_1, 0, -1) $ if $g(\bbeta,x^1_1) \geq \lvert \bbeta_{\text{N}} \rvert$, and $\lambda^\star = \lvert \bbeta_{\text{N}} \rvert$ otherwise. Substituting this optimal value for $\lambda$ into (\ref{eq:counter-example4}), we obtain
\begin{align*}
    f(\bbeta,x_1^1) &= \begin{cases}
        \displaystyle \epsilon \cdot g(\bbeta,x^1_1) + l_{\bbeta}(x^1_1, 0, -1) & \text{if} \;\; \lvert \bbeta_{\text{N}} \rvert \leq g(\bbeta,x^1_1) \\
        \displaystyle \epsilon \cdot \lvert \bbeta_{\text{N}} \rvert + l_{\bbeta}(x^1_1, 0, -1) & \text{otherwise}
    \end{cases}\\
    & = l_{\bbeta}(x^1_1, 0, -1) + h(\bbeta,x^1_1),
\end{align*}
where $$h(\bbeta,x^1_1) := \epsilon \cdot \begin{cases}
        \displaystyle  g(\bbeta,x^1_1)  & \text{if} \;\; \lvert \bbeta_{\text{N}} \rvert \leq g(\bbeta,x^1_1) \\
        \displaystyle  \lvert \bbeta_{\text{N}} \rvert& \text{otherwise.}
    \end{cases} = \epsilon \cdot \max \left\{ g(\bbeta,x^1_1)\,,\, \lvert \bbeta_{\text{N}} \rvert\right\}. $$
The DR logistic regression problem (\ref{eq:counter-example3}) now becomes
\begin{equation*}
    \begin{array}{l@{\quad}l@{\qquad}l}
        \displaystyle \mathop{\text{minimize}}_{\bbeta} & l_{\bbeta}(x^1_1, 0, -1) + h(\bbeta,x^1_1).
    \end{array}
\end{equation*}
Using the definition of the empirical distribution, the regularized logistic regression for the instance problem has the form
\begin{equation}\label{counter-example-regularized}
    \begin{array}{l@{\quad}l@{\qquad}l}
        \displaystyle \mathop{\text{minimize}}_{\bbeta} & l_{\bbeta}(x^1_1, 0, -1) + \mathfrak{R} (\bm{\beta}).
    \end{array}
\end{equation}
The proof concludes by noting that $\mathfrak{R} (\bm{\beta})$ remains constant across all the instances parameterized by $x_1^1$ as a regularizer is data-independent. Hence, the objective function of the regularized logistic regression cannot capture the dependency on $x^1_1$ that we observe in $h(\bbeta,x^1_1)$ (it is easy to confirm that $h(\bbeta,x^1_1)$ is not constant in $x^1_1$). We further notice that the objective function of DR logistic regression is a non-smooth function of $x^1_1$ while the objective function of regularized logistic regression is a smooth function of $x^1_1$ by construction. \qed

\blue{
\begin{rem} \label{rem:NotReg}
It is informative to see how the counter-example in the proof of Theorem \ref{thm:it_aint_regularized} breaks down if we assume there are no categorical features. In this case we use  the ground metric $d (\bxi,\bxi') := \lVert \bx-\bx' \rVert + \kappa \cdot \indicate{y \neq y'}$. The constraints in the problem formulation \eqref{eq:conehead_equiv_formulation} will no longer be indexed over $\bz$ and all terms involving $\bz$ and $d_{\text{C}} (\bz,\bz^i)$ will disappear. The reformulation with only numerical features then coincides with that of Corollary 17 in \cite{JMLR:v20:17-633}. Our instance class problem from the proof of Theorem 3 above will then be
    \begin{equation*}
    \begin{array}{l@{\quad}l@{\qquad}}
        \displaystyle \mathop{\text{minimize}}_{\lambda, s_1,\bbeta} & \displaystyle \lambda \epsilon + s_1 \\
        \displaystyle \text{subject to} & \displaystyle l_{\bbeta}(x^1_1, -1) \leq s_1 \\
        & \displaystyle l_{\bbeta} (x^1_1, 1) - \lambda \kappa \leq s_1  \\
        & \displaystyle \lambda \geq \lvert \bbeta_{\text{\emph{N}}} \rvert, \;\; s_1 \in \mathbb{R}.
    \end{array}
\end{equation*}
Taking $\kappa$ to infinity, we see the second constraint becomes redundant. As $\lambda$ only appears in one term in the objective function, the optimal solution is $\lambda^\star = \lvert \bbeta_{\text{N}} \rvert $. Substituting for $s_1$ in the resulting objective becomes
\begin{equation*}
\begin{array}{l@{\quad}l@{\qquad}l}
        \displaystyle \mathop{\text{minimize}}_{\bbeta} & l_{\bbeta}(x^1_1, -1) + \epsilon \lvert \bbeta_{\text{\emph{N}}} \rvert
    \end{array}
\end{equation*}
which is in the form of regularized logistic regression \eqref{counter-example-regularized}. This of course is consistent with \cite{NIPS2015} who showed that DR logistic regression without categorical features could be formulated as a regularized logistic regression problem.
\end{rem}}

{\bf Proof of Theorem \ref{thm:cac_scheme}:}
It follows from the proof of Theorem~\ref{thm:coneheads} that the optimization problems solved inside the while-loop of Algorithm~\ref{alg:cac_gen} are equivalent to the relaxations
\begin{equation*}
    \begin{array}{l@{\quad}l@{\qquad}l}
        \displaystyle \mathop{\text{minimize}}_{\bm{\beta}, \lambda, \bm{s}} & \displaystyle \lambda \epsilon + \dfrac{1}{N} \sum_{i=1}^N s_i \\[5mm]
        \displaystyle \text{subject to} & \displaystyle l_{\bbeta}(\bx^i, \bz, y^i) - \lambda d_{\text{C}} (\bz,\bz^i)  \leq s_i &  \forall (i, \bz) \in \mathcal{W}^+ \\
        & \displaystyle l_{\bbeta}(\bx^i, \bz, -y^i) - \lambda\kappa - \lambda d_{\text{C}} (\bz,\bz^i)  \leq s_i &  \forall (i, \bz) \in \mathcal{W}^- \\
        & \displaystyle \left \lVert \bm{\beta}_{\text{N}} \right \rVert_* \leq \lambda \\
        & \multicolumn{2}{l}{\displaystyle \mspace{-8mu} \bm{\beta} = (\beta_0, \bm{\beta}_{\text{N}}, \bm{\beta}_{\text{C}}) \in \mathbb{R}^{1 + n + k}, \;\; \lambda \geq 0, \;\; \bm{s} \in \mathbb{R}^N}
    \end{array}
\end{equation*}
of problem~\eqref{eq:conehead_equiv_formulation} that is itself equivalent to our DR logistic regression problem~\eqref{eq:cone_head_reformulation}; the only difference between our relaxations above and problem~\eqref{eq:conehead_equiv_formulation} is that the index sets $(i, \bm{z}) \in [N] \times \mathbb{C}$ in the first two constraint sets of~\eqref{eq:conehead_equiv_formulation} are replaced with the subsets $\mathcal{W}^+$ and $\mathcal{W}^-$ above. This shows that each value $\mathrm{LB}_t$ indeed constitutes a lower bound on the optimal value of problem~\eqref{eq:cone_head_reformulation}, and that the sequence $\{ \mathrm{LB}_t \}_t$ of lower bounds is monotonically non-decreasing. 

To see that the sequence $\{ \mathrm{UB}_t \}_t$ bounds the optimal value of our DR logistic regression problem~\eqref{eq:cone_head_reformulation} from above, note that a maximum constraint violation of $\vartheta^+$ in the first constraint set of~\eqref{eq:cone_head_reformulation} implies that 
\begin{align*}
    \left[
    \begin{array}{l}
        \displaystyle u_{i,\bz}^{+} + v_{i,\bz}^{+} \leq 1 + \vartheta^+ \\
        \displaystyle (u_{i,\bz}^{+}, 1, -s_i - \lambda d_\text{C}(\bz,\bz^i)) \in \mathcal{K}_{\exp} \\
        \displaystyle (v_{i,\bz}^{+}, 1, -y^i\bbeta_{\text{N}}{}^\top \bx^i -y^i\bbeta_{\text{C}}{}^\top \bz - y^i \beta_0 -s_i - \lambda d_\text{C}(\bz,\bz^i)) \in \mathcal{K}_{\exp},
    \end{array}
    \right]
    \quad \forall (i, \bz) \in [N] \times \mathbb{C},
\end{align*}
and, by the definition of the exponential cone, this means that for all $(i, \bz) \in [N] \times \mathbb{C}$, we have
\begin{align*}
    \exp \left[ -s_i - \lambda d_\text{C}(\bz,\bz^i) \right] + \exp \left[ -y^i\bbeta_{\text{N}}{}^\top \bx^i -y^i\bbeta_{\text{C}}{}^\top \bz - y^i \beta_0 -s_i - \lambda d_\text{C}(\bz,\bz^i) \right] \leq 1 + \vartheta^+.
\end{align*}
Dividing both sides by $1 + \vartheta^+ = \exp [\log (1 + \vartheta^+)]$, the inequality becomes
\begin{align*}
    \exp \left[ -s_i' - \lambda d_\text{C}(\bz,\bz^i) \right] + \exp \left[ -y^i\bbeta_{\text{N}}{}^\top \bx^i -y^i\bbeta_{\text{C}}{}^\top \bz - y^i \beta_0 -s_i' - \lambda d_\text{C}(\bz,\bz^i) \right] \leq 1,
\end{align*}
where $s_i' := s_i + \log(1 + \vartheta^+)$. The (potentially) infeasible solution to the relaxed problem of Algorithm~\ref{alg:cac_gen} thus allows us to construct a solution to problem~\eqref{eq:cone_head_reformulation} that satisfies all members of the first constraint set by replacing $s_i$ with $s_i' := s_i + \log(1 + \vartheta^+)$, $i \in [N]$. Compared to the solution of the relaxed problem, the objective value of the newly created solution increased by $\log(1 + \vartheta^+)$. A similar argument can be made for the members of the second constraint set, which shows that a solution to problem~\eqref{eq:cone_head_reformulation} satisfying both constraint sets can be constructed by increasing the objective value by no more than $\log(1 +\max \{\vartheta^+, \vartheta^-\})$. This shows the validity of the upper bound $\theta^\star + \log(1 +\max \{\vartheta^+, \vartheta^-\})$. The fact that the sequence $\{ \mathrm{UB}_t \}_t$ of upper bounds is monotonically non-increasing, on the other hand, holds by construction since we take the minimum of the derived upper bound ($\theta^\star + \log(1 + \max \{\vartheta^+, \vartheta^-\}$) and the upper bound $\mathrm{UB}_{t-1}$ of the previous iteration.

To see that Algorithm~\ref{alg:cac_gen} terminates in finite time, note that every iteration $t$ either identifies a violated constraint $(i, \bm{z}) \in [([N] \times \mathbb{C}) \setminus \mathcal{W}^+] \cup [([N] \times \mathbb{C}) \setminus \mathcal{W}^-]$ that is subsequently added to $\mathcal{W}^+$ and/or $\mathcal{W}^-$ (and will thus never again be identified as violated for the same constraint set), or the update $\mathrm{LB}_t = \mathrm{UB}_t = \theta^\star$ is conducted, in which case the algorithm terminates in the next iteration. Finite termination thus holds since the set $[N] \times \mathbb{C}$ of potentially violated constraints is finite.
\qed

\subsubsection*{\blue{Additional Intuition Behind Algorithm~\ref{alg:most_violated}}} 
In the combinatorial optimization problem (\ref{eq:CombOp1}), the decision variables $\bm{z} \in \mathbb{C}$ only appear in (monotone transformations of) the terms $-y^i\bbeta_{\text{C}}{}^\top \bz$ and $-\lambda d_\text{C} (\bz,\bz^i)$. Moreover, the expression $[d_\text{C} (\bz,\bz^i)]^p$ can only attain one of the values $\delta \in \{ 0, \ldots, m \}$ (\emph{cf.}~Definition~\ref{def:metric}). Conditioning on each possible value of $\delta$, which records the number of categorical features along which $\bz$ and $\bz^i$ disagree, we can therefore maximize the constraint violation by maximizing $-y^i\bbeta_{\text{C}}{}^\top \bz$ along all $\bm{z} \in \mathbb{C}$ that differ from $\bm{z}^i$ in exactly $\delta$ categorical features. Finally, since $-y^i\bbeta_{\text{C}}{}^\top \bz$ is linearly separable, that is, $-y^i\bbeta_{\text{C}}{}^\top \bz = -y^i\bbeta_{\text{C}, 1}{}^\top \bz_1 - \ldots - y^i\bbeta_{\text{C}, m}{}^\top \bz_m$, we can choose the categorical features along which $\bm{z}$ and $\bm{z}^i$ disagree iteratively by considering each expression $-y^i\bbeta_{\text{C}}{}^\top \bz$ separately and choosing the disagreeing feature values greedily based on their individual contribution to the overall constraint violation. \medskip

{\bf Proof of Theorem~\ref{thm:algo_complex}}
Recall that Algorithm~\ref{alg:most_violated} aims to solve the optimization problem
\begin{equation*}
    \mspace{-25mu}
    \begin{array}{l@{\quad}l}
        \displaystyle \mathop{\text{maximize}}_{\bm{z}} & \displaystyle \min_{u^{+}, v^{+}} \left\{ u^{+} +  v^{+} \, : \, \left[ \begin{array}{l}
            \displaystyle (u^{+}, 1, -s_i - \lambda d_\text{C} (\bz,\bz^i)) \in \mathcal{K}_{\exp}  \\
            \displaystyle (v^{+}, 1, -y^i\bbeta_{\text{N}}{}^\top \bx^i -y^i\bbeta_{\text{C}}{}^\top \bz - y^i \beta_0 -s_i - \lambda d_\text{C} (\bz,\bz^i)) \in \mathcal{K}_{\exp}
        \end{array}
        \right] \right\} \\[4.5mm]
        \displaystyle \text{subject to} & \displaystyle \bm{z} \in \mathbb{C}.
    \end{array}
\end{equation*}
It immediately follows from the definition of $\mathcal{K}_{\exp}$ that the above problem can be represented as
\begin{equation*}
    \begin{array}{l@{\quad}l}
        \displaystyle \mathop{\text{maximize}}_{\bm{z}} & \displaystyle \exp \left[ -s_i - \lambda d_\text{C} (\bz,\bz^i) \right] + \exp \left[ -y^i\bbeta_{\text{N}}{}^\top \bx^i -y^i\bbeta_{\text{C}}{}^\top \bz - y^i \beta_0 -s_i - \lambda d_\text{C} (\bz,\bz^i) \right] \\
        \displaystyle \text{subject to} & \displaystyle \bm{z} \in \mathbb{C}.
    \end{array}
\end{equation*}
In this problem, the optimization variable $\bm{z}$ appears inside the exponential terms as $-y^i\bbeta_{\text{C}}{}^\top \bz$ and as $-\lambda d_\text{C} (\bz,\bz^i)$. Since increasing the value of one term may decrease the value of the other term, the above problem does not admit a trivial solution. We next discuss how the above maximization problem can be decomposed into $m+1$ sub-problems, each of which can be solved efficiently.

Although $\bz$ may take exponentially many values, notice that the expression $d_\text{C} (\bz,\bz^i) = ( \sum_{j \in [m]} \indicate{\bm{z}_j \neq \bm{z}^i_j})^{1 / p}$ counts the number of disagreements between the features of $\bm{z}$ and $\bm{z}^i$, which takes a value from the set $\{0,1,\ldots, m\}$. Hence, Algorithm~\ref{alg:most_violated} decomposes the above combinatorial optimization problem into $m+1$ sub-problems in which the number $\delta \in \{0,1, \ldots,m \}$ of disagreements between $\bm{z}$ and $\bm{z}^i$ is fixed and we only need to maximize $-y^i\bbeta_{\text{C}}{}^\top \bz $ for a given number $\delta$ of disagreements. If we can solve each of these sub-problem efficiently, then we can compare the constraint violations of the optimal $\bz$ to each sub-problem and pick the largest one as the most-violated constraint. We next investigate how each of the $m+1$ sub-problems can be solved efficiently.

The sub-problem corresponding to a fixed $\delta \in \{0,1,\ldots,m\}$ can be formulated as
\begin{align*}
    \begin{array}{cl}
    \underset{\bz}{\mathrm{maximize}}  &  -y^i\bbeta_{\text{C}}{}^\top \bz \\
    \mathrm{subject\;to}  & d_\text{C} (\bz,\bz^i) = \delta \\
    & \bz \in \mathbb{C}.
    \end{array}
\end{align*}
Denote by $\mathcal{J} \subseteq [m]: \ |\mathcal{J}| = \delta$ the set of features where $\bm{z}$ disagrees with $\bz^i$ at optimality. Denote further by $\bz_j^\star$ the maximizer of $- y^i \cdot \bbeta_{\text{C},j}{}^\top \bz_j$ over $\bz_j \in \mathbb{C}(k_j)\setminus \{ \bz^i_j \}$ (\emph{cf.}~the first step of Algorithm~\ref{alg:most_violated}). The above sub-problem can then be written as 
\begin{align*}
    \begin{array}{cl}
    \underset{\mathcal{J}}{\mathrm{maximize}}  & \displaystyle \left(\sum_{j = 1}^m  - y^i \cdot \bbeta_{\text{C},j}{}^\top \bz_j^i\right) + \left(\sum_{j \in \mathcal{J}} (- y^i \cdot \bbeta_{\text{C},j}{}^\top \bz_j^\star)  -  (- y^i \cdot \bbeta_{\text{C},j}{}^\top \bz_j^i) \right) \\[5mm]
    \mathrm{subject\;to}  & \mathcal{J} \subseteq [m],\; |\mathcal{J}|  = \delta.
    \end{array}
\end{align*}
The first summation above is constant, and the solution of this problem can therefore be uniquely identified from the solution of the following problem:
\begin{align*}
    \begin{array}{cl}
    \underset{\mathcal{J}}{\mathrm{maximize}}  & \displaystyle \sum_{j \in \mathcal{J}} (- y^i \cdot \bbeta_{\text{C},j}{}^\top (\bz_j^\star - \bz_j^i))  \\[5mm]
    \mathrm{subject\;to}  & \mathcal{J} \subseteq [m],\; |\mathcal{J}|  = \delta.
    \end{array}
\end{align*}
This problem can be solved greedily by selecting the $\delta$ largest values  of $(- y^i \cdot \bbeta_{\text{C},j}{}^\top (\bz_j^\star - \bz_j^i))$ across all $j \in [m]$, and those $j$ indexes will be included in $\mathcal{J}$ (\emph{cf.}~the ordering $\pi: [m] \mapsto [m]$ and the selection according to $\pi(j) \leq \delta$ in Algorithm~\ref{alg:most_violated}). Instead of sorting $(- y^i \cdot \bbeta_{\text{C},j}{}^\top (\bz_j^\star - \bz_j^i))$, $j \in [m]$, for each sub-problem $\delta \in \{ 0,1,\ldots,m \}$, Algorithm~\ref{alg:most_violated} sorts those values once in the beginning.

The discussion so far establishes the correctness of Algorithm~\ref{alg:most_violated}. In view of its runtime, we note that the aforementioned sorting takes time $\mathcal{O}(m \log m)$ as each value $(- y^i \cdot \bbeta_{\text{C},j}{}^\top (\bz_j^\star - \bz_j^i))$, $j \in [m]$ can be computed in constant time, determining $\bz_j^\star$ for all $j \in [m]$ takes time $\mathcal{O}(k)$ since the variables corresponding to each feature vanish in all but one (known) location, and determining the set $\mathcal{J}$ for each $\delta \in \{ 0,1,\ldots,m\}$ takes time $\mathcal{O} (n + m^2)$ since the expression $-y^i \bm{\beta}_\text{N}{}^\top \bm{x}^i$ only needs to be computed once.
\qed

{\bf Proof of Theorem \ref{thm:finite_sample_guarantee}:}
The statement of the theorem follows immediately from Theorems~18 and~19 of \cite{kmns19}.
\qed

{\bf Proof of Lemma~\ref{lem:growth_condition}:}
The assumption in the statement of the lemma allows us to choose $M \in \mathbb{R}$ such that $\mathcal{H} \subseteq [-M, +M]^{1 + n + k}$. We then show that there is $\bm{\xi}^0 \in \Xi$ and $C > 0$ such that
\begin{equation*}
    l_{\bm{\beta}} (\bm{\xi}) \leq C [1 + d (\bm{\xi}, \bm{\xi}^0)] \quad \forall \bm{\beta} \in [-M, M]^{1 + n + k}, \; \forall \bm{\xi} \in \Xi,
\end{equation*}
that is, for all $\bm{\beta} \in [-M, M]^{1 + n + k}$ and all $\bm{\xi} \in \Xi$ we have
\begin{align*}
    & \log \left( 1 + \exp \left[ -y \cdot \left( \beta_0 + \bm{\beta}_{\text{N}}{}^\top \bm{x} + \bm{\beta}_{\text{C}}{}^\top \bm{z} \right) \right] \right) \\
    \leq \;\; &
    C \left[ 1 +  \lVert \bx-\bx^0 \rVert + \left( \sum_{i \in [m]} \indicate{\bm{z}_i \neq \bm{z}^0_{i}} \right)^{1 / p} + \kappa \cdot \indicate{y \neq y^0} \right].
\end{align*}
To see this, fix any $\bxi^0 = (\mathbf{0}, \bz^0, y^0) \in \Xi$. Since $\kappa \cdot \indicate{y \neq y^0} \geq 0$ and $ \sum_{i \in [m]} \indicate{\bm{z}_i \neq \bm{z}^0_{i}} )^{1 / p} \geq 0$ on the right-hand side of the above inequality, it suffices to show that
\begin{equation*}
  \log \left( 1 + \exp \left[ -y \cdot \left( \beta_0 + \bm{\beta}_{\text{N}}{}^\top \bm{x} + \bm{\beta}_{\text{C}}{}^\top \bm{z} \right) \right] \right) \leq C \left[1 +  \lVert \bx \rVert \right] \quad  \forall \bm{\beta} \in [-M, M]^{1 + n + k}, \; \forall \bm{\xi} \in \Xi.
\end{equation*}
Note that the left-hand side of this inequality can be bounded from above by
\begin{align*}
    1 + \exp \left[ -y \cdot \left( \beta_0 + \bm{\beta}_{\text{N}}{}^\top \bm{x} + \bm{\beta}_{\text{C}}{}^\top \bm{z} \right) \right]
    \;\; \leq \;\; &
    1 + \exp \left[ \left| \beta_0 + \bm{\beta}_{\text{N}}{}^\top \bm{x} + \bm{\beta}_{\text{C}}{}^\top \bm{z} \right| \right] \\
    \leq \;\; &
    2\exp \left[\left|\beta_0 + \bm{\beta}_{\text{N}}{}^\top \bm{x} + \bm{\beta}_{\text{C}}{}^\top \bm{z} \right| \right].
\end{align*}
It is therefore sufficient to show that
\begin{equation*}
    \mspace{-20mu}
    \begin{array}{l@{\;\;}l@{\qquad}l}
        & \displaystyle
        \log(2\exp \left[\left|\beta_0 + \bm{\beta}_{\text{N}}{}^\top \bm{x} + \bm{\beta}_{\text{C}}{}^\top \bm{z} \right| \right]) \leq C[1 +  \lVert \bx \rVert]
        & \displaystyle
        \forall \bm{\beta} \in [-M, M]^{1 + n + k}, \\
        & & \forall \bm{\xi} \in \Xi \\
        \displaystyle \iff
        & \displaystyle
        \log(2) + \left|\beta_0 + \bm{\beta}_{\text{N}}{}^\top \bm{x} + \bm{\beta}_{\text{C}}{}^\top \bm{z} \right| \leq C[1 + \lVert \bx \rVert]
        & \displaystyle
        \forall \bm{\beta} \in [-M, M]^{1 + n + k}, \\
        & & \forall \bm{\xi} \in \Xi \\
        \displaystyle \iff
        & \displaystyle
        \log(2) + \max_{\bm{\beta} \in [-M, M]^{1 + n + k}} \left\{ \left|\beta_0 + \bm{\beta}_{\text{N}}{}^\top \bm{x} + \bm{\beta}_{\text{C}}{}^\top \bm{z} \right| \right\} \leq C[1 + \lVert \bx \rVert]
        & \displaystyle \forall \bm{\xi} \in \Xi \\
        \displaystyle \iff
        & \displaystyle \log(2) + (M + M \lVert \bx \rVert + M k) \leq C[1 + \lVert \bx \rVert]
        & \displaystyle
        \forall \bm{\xi} \in \Xi \\
        \displaystyle \iff
        & \displaystyle
        \log(2) + M(1 + k) + M\lVert \bx \rVert \leq C + C\lVert \bx \rVert
        & \displaystyle
        \forall \bm{\xi} \in \Xi.
    \end{array}
\end{equation*}
The last condition, however, is readily seen to be satisfied by any $C \geq \log(2) + M(1+ k)$.
\qed

{\bf Proof of Theorem \ref{thm:asymptotic_consistency}:}
The statement of the theorem follows directly from Theorem~20 of \cite{kmns19}.
\qed

{\bf Proof of Theorem \ref{thm:existence_sparse_distributions}:}
The existence of an optimal solution follows from Theorem~3 of \cite{YKW21:linear_optimization_wasserstein}, and the sparsity of the optimal solution is due to Theorem~4 of \cite{YKW21:linear_optimization_wasserstein}. Note that the assumptions of those theorems are satisfied since the Wasserstein ball $\mathfrak{B}_\epsilon (\widehat{\mathbb{P}}_N)$ is centered at the empirical distribution $\widehat{\mathbb{P}}_N$, which has finite moments by construction.
\qed

\end{document}